\documentclass[11pt]{article}
\usepackage[margin=1in]{geometry}
\usepackage{amsmath,amssymb,amsthm,mathtools}
\usepackage{enumitem}
\usepackage{graphicx}
\usepackage[hidelinks]{hyperref}
\usepackage{xcolor}
\newtheorem{theorem}{Theorem}
\newtheorem{proposition}[theorem]{Proposition}
\newtheorem{lemma}[theorem]{Lemma}

\theoremstyle{definition}
\newtheorem{assumption}{Assumption}
\newtheorem{remark}{Remark}

\newcommand{\R}{\mathbb R}
\newcommand{\E}{\mathbb E}
\newcommand{\Prob}{\mathbb P}
\newcommand{\norm}[1]{\left\lVert #1\right\rVert}
\newcommand{\eps}{\varepsilon}
\newcommand{\Tr}{\operatorname{Tr}}
\newcommand{\diag}{\operatorname{diag}}
\newcommand{\Lc}{\mathcal L}
\newcommand{\kstop}{C_{\mathrm{stop}}}
\newcommand{\tprior}{t_{\mathrm{prior}}}
\newcommand{\ind}{\mathbf 1}
\DeclareMathOperator{\rank}{rank}

\title{Minimax-Optimal Early Stopping for SGD via the Discrepancy Principle}
\author{Tim Jahn\thanks{TU Berlin, Germany, jahn@tu-berlin.de}
\and Loucas Pillaud-Vivien\thanks{CERMICS, CNRS, ENPC, Institut Polytechnique
de Paris, Marne-la-Vall\'ee, France,\\ loucas.pillaud-vivien@enpc.fr}
\and Adrien Schertzer\thanks{Department of Mathematics, ENS de Lyon and CNRS,
Lyon, France, schertzer@hotmail.fr}}
\date{\today}

\begin{document}
\maketitle

\begin{abstract}
We study early stopping for stochastic gradient descent with a constant
step size applied to noisy linear inverse problems with polynomially decaying spectrum. The iteration is stopped by the discrepancy
principle, that is, as soon as its residual reaches the prescribed noise
level. For H\"older source conditions below a saturation level
determined by the polynomial decay, we prove that the mean-square
reconstruction error at the stopping index is bounded by the classical
rate, up to an explicit factor involving the step size and the
incoherence of the design, and that the expected number of iterations is
of the order of the oracle stopping index. The rate is minimax-optimal
for deterministic noise whenever the spectrum contains an eigenvalue at
the oracle scale and this factor stays bounded. The step size
is restricted by a spectrum-weighted incoherence condition. The analysis
allows rank-deficient designs and inconsistent observations.
\end{abstract}



\section{Introduction}\label{sec:intro}

Early stopping is a classical regularization method for ill-posed inverse
problems \cite{engl1996,hansen1998rank,vogel2002computational}, where the
stopping index of an iterative solver determines the effective level of
regularization. It is also a widely used form of implicit regularization
in machine learning, from boosting
\cite{buhlmann2003boosting,zhang2005boosting} and gradient descent in
reproducing kernel Hilbert spaces \cite{yao2007early,raskutti2014early}
to the training of large-scale networks and of generative or untrained
networks used as reconstruction priors \cite{jahn2024early,martin2025pnp}.
In all of these settings, prolonged optimization may amplify noise,
overfit finite samples, or degrade generalization. In this work we study
early stopping for stochastic gradient descent with a constant step size
applied to a noisy linear inverse problem, and ask whether an a posteriori
stopping rule remains rate-optimal in the presence of stochastic-gradient
fluctuations.

\paragraph{Setting.}
We seek to recover $\theta^\star\in\R^d$ from noisy linear observations
\[
 y=X\theta^\star+\eps,\qquad y,\eps\in\R^n,\qquad
 X=\begin{pmatrix}x_1^\top\\ \vdots\\ x_n^\top\end{pmatrix}\in\R^{n\times d},
\]
with $\norm\eps\le\sqrt n\,\delta$ for a known noise level $\delta>0$.
The empirical covariance
$\Sigma:=X^\top X/n$ and the empirical kernel $K:=XX^\top/n$ share their
nonzero eigenvalues $\lambda_1\ge\dots\ge\lambda_m>0$, counted with
multiplicity, where $m:=\rank X$. The design $X$, which plays the role
of the forward operator, the observation noise, and the initialization
are deterministic; all probabilities and expectations
refer to the sampled row indices. All vector norms are Euclidean.

A standard approach is to minimize the empirical least-squares functional
\[
 F(\theta):=\frac1{2n}\norm{X\theta-y}^2=\frac1n\sum_{i=1}^nF_i(\theta),
 \qquad F_i(\theta):=\tfrac12\bigl(\langle x_i,\theta\rangle-y_i\bigr)^2 .
\]
Stochastic gradient descent (SGD) minimizes this objective using one row
at a time \cite{robbins1951stochastic,bottou2018}. Starting from a
deterministic $\theta_0\in\R^d$ and with a constant step size
$\gamma>0$, it draws independently at each step $k$ an index $i_k$ uniformly from
$\{1,\dots,n\}$ and sets
\begin{equation}\label{eq:sgd}
 \theta_{k+1}=\theta_k-\gamma\nabla F_{i_k}(\theta_k)
 =\theta_k-\gamma\,x_{i_k}(x_{i_k}^\top\theta_k-y_{i_k}).
\end{equation}
For rows of equal norm, \eqref{eq:sgd} is the randomized Kaczmarz
iteration \cite{strohmer2009randomized,needell2016stochastic} with
relaxation parameter $\gamma\norm{x_i}^2$; in general it belongs to the
randomized iterative solvers for linear systems analyzed in
\cite{gower2015randomized}. Such stochastic and randomized row-action
methods are increasingly used for large-scale inverse problems, where even
a single pass over the data is expensive, for instance in tomographic and
emission imaging
\cite{chambolle2018stochastic,ehrhardt2019faster,tang2020practicality,bubba2026linesearch};
see \cite{ehrhardt2025guide} for a survey.

\paragraph{Interpolation and early stopping.}
To see why the iteration has to be stopped, consider its limit. For
constant step sizes as in \ref{ass:step} below and a consistent linear
system, the iteration \eqref{eq:sgd} converges to the interpolating
solution selected by its initialization. Consistency means that the interpolation set
$\mathcal I_y:=\{\theta\in\R^d:X\theta=y\}$ is nonempty; this holds for
every $y$ exactly when $X$ has full row rank, $m=n$, the generic situation
in the overparameterized regime $d\ge n$. Let $\Pi$ denote the orthogonal
projection onto the null space $\mathcal N(X)$ and $X^\dagger$ the
Moore--Penrose pseudo-inverse of $X$. The limit is the orthogonal projection of $\theta_0$
onto $\mathcal I_y$ \cite{schertzer2024stochastic},
\[
 \theta_\infty=\operatorname*{argmin}_{\theta\in\mathcal I_y}
 \norm{\theta-\theta_0}=X^\dagger y+\Pi\theta_0,
 \qquad\text{hence}\qquad
 \theta_\infty-\theta^\star=X^\dagger\eps+\Pi(\theta_0-\theta^\star).
\]
When $X$ is ill-conditioned, even a small observation noise may thus lead
to a large reconstruction error. Averaged over the row sampling, the
iteration fits the singular modes of $X$ at rates given by the squared
singular values, so the directions in which $X^\dagger$ amplifies the
noise most are fitted last.
The iteration must therefore be regularized, classically by early
stopping: it is interrupted once the stable signal components are
recovered and before the amplified noise components are fitted
\cite{engl1996,hansen1998rank,vogel2002computational,yao2007early}.%
\footnote{Interpolating noisy data is not always harmful in the
overparameterized regime \cite{bartlett2020benign}, but
benign-overfitting regimes typically rely on spectral tails heavier than
those considered here.}

\paragraph{The discrepancy principle.}
When the noise level $\delta$ is known, the classical noise-calibrated
stopping rule is the discrepancy principle
\cite{morozov1966solution,engl1996}. It tracks the observable residual
$r_k:=X\theta_k-y$, linked to the error $\eta_k:=\theta_k-\theta^\star$ by
$r_k=X\eta_k-\eps$, and stops as soon as the residual norm reaches the
scale of the data noise,
\begin{equation}\label{eq:dp}
 \tau:=\inf\bigl\{k\in\mathbb N_0:\norm{r_k}\le\kstop\sqrt n\,\delta\bigr\},
 \qquad\inf\varnothing=+\infty,
\end{equation}
for a safety factor $\kstop>1$. The rule uses the noise level but does
not require knowledge of the smoothness of $\theta^\star$.
When $m<n$, the system may be inconsistent: the component of the residual
in $\mathcal N(X^\top)$ equals that of $-\eps$ for every $k$ and is never
reduced, and the iteration need not converge \cite{needell2010randomized,ma2015convergence,zouzias2013randomized}.
The rule can still stop, since this component has norm at most
$\norm\eps\le\sqrt n\,\delta$, below the threshold.

The test needs the full residual, which obeys the recursion
\begin{equation}\label{eq:residual}
 r_{k+1}=r_k-\gamma\,(r_k)_{i_k}\,Xx_{i_k}
 =r_k-\gamma nKe_{i_k}e_{i_k}^\top r_k ,
\end{equation}
since $x_{i_k}^\top\theta_k-y_{i_k}=(r_k)_{i_k}$; here $e_i$ is the $i$th
canonical vector of $\R^n$. After a one-time computation of $K$ at cost
$O(n^2d)$, each step thus updates $r_k$ at cost $O(n)$. The results extend
to linear predictors in a reproducing kernel Hilbert space, where $K$ is
the Gram matrix and is formed in any case. When $n\gg d$, one may instead
check the residual only every $c$ steps, for instance once per epoch, as
in the restarted method of \cite{zhu2026restarted}; the proof extends
with minor changes as long as $c\gamma\le\tprior$, with the oracle time
$\tprior$ of \eqref{eq:scales} below. Random
reshuffling, mini-batches, or estimating the residual from a subsample
require further work.

\paragraph{Assumptions.}
We assume $m\ge1$. If $X=0$, the source condition below implies
$\theta_0=\theta^\star$ and the discrepancy principle stops immediately.
Choose orthonormal singular vectors $v_j\in\R^d$,
$u_j\in\R^n$ such that
\[
 \Sigma v_j=\lambda_jv_j,\qquad
 Xv_j=\sqrt{n\lambda_j}\,u_j,\qquad j\le m.
\]
Complete $(u_j)$ to an orthonormal basis of $\R^n$, and put $\lambda_j=0$
for $j>m$. Fix this eigenbasis once and for all, including within
eigenspaces of repeated eigenvalues. For $j\le m$, define $\mu_j^2$
and the aggregate quantities
\begin{equation}\label{eq:incoherence}
 \mu_j^2=n\max_{i\le n}u_{j,i}^2,\qquad
 \mu^2=\max_{j\le m}\mu_j^2,\qquad
 \kappa=\Tr K,\qquad
 \mu_\star^2=\frac{\sum_{j\le m}\mu_j^2\lambda_j}{\kappa},\qquad
 \nu=\frac\gamma2\sum_{j\le m}\mu_j^2\lambda_j.
\end{equation}
Here $\kappa=\frac1n\sum_i\norm{x_i}^2$ is the average squared row norm,
while $\mu^2$ and $\mu_\star^2$ measure how the singular vectors of $X$
spread over the samples, uniformly and weighted by the spectrum, in the
spirit of the incoherence conditions of compressed sensing and matrix
completion \cite{candes2007sparsity,candes2009exact}. Since each $u_j$ is
a unit vector, $1\le\mu_\star^2\le\mu^2\le n$.

\begin{assumption}\label{ass:main}
The following conditions are assumed throughout the analysis:
\begin{enumerate}[label=(A\arabic*),ref=(A\arabic*),leftmargin=*]
\item\label{ass:spectrum} $\lambda_j\le C_\lambda j^{-\beta}$ for $j\le m$, with
      $C_\lambda>0$ and $\beta>1$.
\item\label{ass:source} $\eta_0=\Sigma^r v$, where $r>0$, $\norm v\le\rho$, and $\rho>0$.
\item\label{ass:compatibility} $2r<(\beta-1)/\beta$.
\item\label{ass:step} $\nu\le1/4$.
\end{enumerate}
\end{assumption}

The spectral bound is the usual capacity condition in kernel learning
\cite{caponnetto2007optimal}, and the source condition expresses
smoothness relative to the design \cite{engl1996}. The uniform
incoherence $\mu^2$ enters the error bounds, while the weighted
incoherence $\mu_\star^2$ controls the admissible step size through
$\nu=\gamma\mu_\star^2\kappa/2$. Condition \ref{ass:compatibility}, the
compatibility condition between the smoothness exponent and the spectral
decay, is a saturation-type restriction; it is discussed after the
theorem.

\paragraph{Main result.}
Define the oracle time and error scales
\begin{equation}\label{eq:scales}
 \tprior=\Bigl(\frac\rho\delta\Bigr)^{2/(2r+1)},\qquad
 M=\rho\tprior^{-r}=\delta\sqrt{\tprior}.
\end{equation}
The oracle time balances the approximation error
$\rho t^{-r}$ with the propagated-noise scale $\delta\sqrt t$, and
$M$ is their common value. Iteration $k$ corresponds to time
$k\gamma$; this convention will be used throughout.

\begin{theorem}[Discrepancy principle for discrete SGD]\label{thm:main}
Suppose Assumption~\ref{ass:main} holds and $\kstop>1/(1-\nu)$, and put
$g_{\mathrm d}=(1-\nu)\kstop-1>0$. Assume $\delta\le\rho/e$ and
$\tprior\ge\gamma$.
Then $\tau<\infty$ almost surely, and
\begin{equation}\label{eq:main-expectation}
 \E\norm{\theta_\tau-\theta^\star}^2
 \le\bigl(C_E+\widetilde C_E\gamma\mu^2\bigr)\rho^{2/(2r+1)}\delta^{4r/(2r+1)},
 \qquad
 \gamma\,\E\tau\le\bigl(C_\tau+\widetilde C_\tau\gamma\mu^2\bigr)\tprior ,
\end{equation}
where $C_E,\widetilde C_E,C_\tau$ depend only on $r,\beta,C_\lambda,\kstop$,
and $\widetilde C_\tau=C_\tau'+C_\tau''/g_{\mathrm d}$ with
$C_\tau',C_\tau''$ depending only on the same quantities. All are given
explicitly in Appendix~\ref{sec:constants}.
\end{theorem}

\paragraph{Interpretation.} The root-mean-square error is at most
$(C_E+\widetilde C_E\gamma\mu^2)^{1/2}M$. The minimax comparison below
shows that this rate is optimal up to a constant whenever $K$ has an
eigenvalue in $[c\,\tprior^{-1},\tprior^{-1}]$ for a fixed $c\in(0,1]$
and $\gamma\mu^2$ stays bounded. Under the same control of the incoherence
factor, the expected number of iterations is at most a constant multiple
of the oracle scale $\tprior/\gamma$. Neither $r$ nor
$\rho$ is used by the iteration or the stopping rule; they enter only the
proof, which localizes the stopping index on blocks of length of order
$\tprior/\gamma$. The condition $\tprior\ge\gamma$ controls rounding to
the integer grid and holds eventually as $\delta/\rho\to0$ for any fixed
admissible step.

The condition $\kstop>1/(1-\nu)$ tends to the condition $\kstop>1$
of Landweber iteration as the step size, and with it $\nu$, tends to zero. Some excess over one is needed at a
constant step size, because the sampling noise keeps the residual above
the level of the noise that has not yet been fitted.
Remark~\ref{rem:threshold} discusses the thresholds, and
Section~\ref{sec:numerics} illustrates how a safety factor merely above
one can lead to much later stopping or to failure to stop within the
simulation budget.

The two terms in \eqref{eq:main-expectation} separate the deterministic
part of the analysis from the sampling noise. The terms without
$\gamma\mu^2$ are of the same type as for Landweber iteration: they
depend only on $r$ and $\kstop$, and as $\kstop\downarrow1$ they grow
like $(\kstop-1)^{-2/(2r+1)}$ up to a factor $\log(1/(\kstop-1))$, which
is the growth produced by the classical argument for the discrepancy
principle \cite{engl1996}; see Appendix~\ref{sec:constants}. Apart from
the threshold $1/(1-\nu)$, the sampling noise enters only through the
terms proportional to $\gamma\mu^2$. Our bounds for these grow faster as
$\kstop\downarrow1$, and for the number of iterations they contain the
factor $g_{\mathrm d}^{-1}$, which blows up as $\kstop$ approaches
$1/(1-\nu)$. As $\gamma\mu^2\to0$, the bounds reduce to bounds of
Landweber type.

The compatibility condition $2r<1-1/\beta$ reflects a finite-qualification
phenomenon with counterparts in earlier analyses of SGD with decaying step
sizes for inverse problems \cite{jin2019regularizing,jahn2020discrepancy}. For decaying microscopic step sizes, of
order $n^{-1}$, it might be removable: \cite{jin2021saturation} obtain the optimal rate
without saturation for the step sizes $\gamma_j=c_0j^{-a}$, $a\in[0,1)$,
of \cite{jin2019regularizing,jahn2020discrepancy}, with a microscopic
initial step $c_0=O(n^{-1})$ and a stopping index calibrated to the
unknown regularity. The experiments in \cite[Section~3.2]{jahn2026sde}
show the loss of accuracy when the condition is violated.

\paragraph{Minimax comparison.}
We now specify when the scale
$M=\rho^{1/(2r+1)}\delta^{2r/(2r+1)}$ is minimax-optimal for
deterministic noise. A reconstruction is a measurable map
\[
 \widehat\theta:\R^n\times\Xi\to\R^d,\qquad
 (y,\xi)\mapsto\widehat\theta(y,\xi),
\]
evaluated at the data $y$ and at an auxiliary random variable $\xi$ with
values in a measurable space $\Xi$. The map and the law of $\xi$ may
depend on $X$, $\theta_0$, $r$, $\rho$ and $\delta$. Deterministic
reconstructions do not depend on $\xi$, and the stopped SGD iterate
$\theta_\tau$, set to $\theta_0$ on $\{\tau=\infty\}$, is a
reconstruction with $\xi=(i_k)_{k\ge0}$. With $\E$ the expectation over
$\xi$, the minimax risk is
\[
 \mathcal R_\delta:=\inf_{\widehat\theta}\ \sup_{\norm v\le\rho}\
 \sup_{\norm\eps\le\sqrt n\,\delta}
 \bigl(\E\norm{\widehat\theta(y,\xi)-\theta^\star}^2\bigr)^{1/2},
 \qquad
 \theta^\star=\theta_0-\Sigma^rv,\quad y=X\theta^\star+\eps,
\]
where the infimum runs over all reconstructions. Then
\begin{equation}\label{eq:minimax}
 c^rM\le\mathcal R_\delta\le2M ,
\end{equation}
where the upper bound holds for every design and the lower bound
whenever $K$ has an eigenvalue in $[c\,\tprior^{-1},\tprior^{-1}]$ for
some $c\in(0,1]$. For $\lambda_j=j^{-\beta}$, $j\le m$, this is the case
with $c=2^{-\beta}$ as long as $1\le\tprior\le m^\beta$. For deterministic
reconstructions, \eqref{eq:minimax} is the classical optimal-recovery
bound of regularization theory
\cite{micchelli1977survey,melkman1979optimal},
\cite[Section~3.2]{engl1996}, formulated there for operators between
Hilbert spaces of any dimension. Its lower bound rests on two admissible
signals $\theta_\pm=\theta_0\pm\rho\lambda_j^rv_j$ with
$\lambda_j\le\tprior^{-1}$ that produce the same data. It therefore extends to randomized
reconstructions in mean square, since
$\norm{z-\theta_+}^2+\norm{z-\theta_-}^2\ge2\rho^2\lambda_j^{2r}$ for every
$z$. In the terminology of regularization theory, Theorem~\ref{thm:main}
thus states that the stopped iteration is of optimal order on the source
class \cite[Section~3.2]{engl1996}, in mean square and for bounded
$\gamma\mu^2$. The eigenvalue condition cannot be dropped. For a fixed design of rank $m$,
the reconstruction $\theta_0+X^\dagger(y-X\theta_0)$ has error
at most $\delta/\sqrt{\lambda_m}$, which is smaller than
$M=\delta\sqrt{\tprior}$ once $\tprior>\lambda_m^{-1}$.

\paragraph{Related work.}
Several lines of work bear on this problem, but none yields the guarantee
obtained here. First, stopping an iterative scheme at a data-driven index
is classical in inverse problems \cite{engl1996} and in statistics, where
it has been analyzed for boosting
\cite{buhlmann2003boosting,zhang2005boosting}, for gradient descent in
reproducing kernel Hilbert spaces \cite{yao2007early,raskutti2014early},
for least squares in continuous time \cite{ali2019continuous}, and for
statistical inverse problems, where sharp adaptation results are available
\cite{blanchard2018optimal}. The noise-calibrated rule we use goes back to
Morozov \cite{morozov1966solution}; for the deterministic iteration it is
adaptive and rate-optimal \cite{engl1996}, and its statistical
counterparts have been studied for spectral filters and conjugate
gradients
\cite{blanchard2012discrepancy,blanchard2018optimal,celisse2021analyzing}.
The rule needs the noise level: under a worst-case analysis no rule can be
convergent without it \cite{bakushinskii1984remarks}, whereas under a
stochastic noise model noise-level-free variants exist
\cite{becker2011regularization,harrach2020beyond,jahn2023noise,harrach2023regularizing}.
All of these results concern deterministic dynamics.

SGD and its variance-reduced variants are known to act as regularization
methods, with convergence rates established under various step-size
schedules and for a priori chosen regularization parameters
\cite{jin2019regularizing,jin2020convergence,lu2022stochastic,jinlutian2026,jin2022svrg,jinzhou2026svrg};
see \cite{jin2025survey,ehrhardt2025guide} for recent surveys. Where the
discrepancy principle has been considered for plain SGD, the guarantees
concern termination and convergence, but do not quantify the error
attained at the stopping index \cite{jahn2020discrepancy}. The same holds
for stochastic mirror descent \cite{huang2025early}, for an iteratively
regularized variant of SGD under the noise-level-free Hanke--Raus rule
\cite{bajpai2026hankeraus}, for an adaptive stochastic heavy ball method
\cite{gu2026posteriori}, and for a restarted adaptive stochastic Nesterov
method that evaluates the rule only periodically \cite{zhu2026restarted}. For the variance-reduced variant SVRG,
termination was likewise established in \cite{jinqchen2025svrg}, and,
recently, \cite{jinzhou2026dp} obtained convergence
rates for the rule evaluated at the anchor points of the
variance-reduction scheme. Variance reduction changes the structure of the
noise: there the gradient noise is proportional to the displacement from
the current anchor and decays along the iteration, whereas for plain SGD
it is proportional to the residual and persists at the noise level, which
is the difficulty addressed here. The estimates of \cite{jinzhou2026dp}
are stated for a fixed
source element, with constants depending on it, on the sample size and on
the update frequency; the bounds of Theorem~\ref{thm:main} are uniform over the source class
\ref{ass:source}, with constants independent of $n$ and $d$, which is what
permits the comparison with the minimax rate over that class.

On the learning-theory side it is known that multi-pass SGD attains
minimax rates under capacity conditions with the number of passes as the
regularization parameter at fixed step size
\cite{lin2017optimal,pillaudvivien2018statistical}, and averaged
single-pass SGD is optimal at large step sizes
\cite{dieuleveut2016nonparametric} in the tradition of
\cite{caponnetto2007optimal}. There the tuned quantity is calibrated a
priori to $n$ and to the unknown exponents. In neither strand is the
residual used to decide when to stop. The high-dimensional theory of
\cite{paquette2021sgd,paquette2025homogenization} describes the loss of
SGD exactly, but for random data and without any stopping rule.
Diffusion approximations of SGD are classical in learning
\cite{li2017stochastic,li2019stochastic,sirignano2017stochastic}, and a
diffusion whose covariance matches that of the stochastic gradient
exactly retains its multiplicative noise
\cite{ali2020implicit,schertzer2024stochastic}. For this diffusion, an
earlier version of this article \cite{jahn2026sde} shows that the
discrepancy principle attains the minimax rate with high probability, up
to a logarithmic factor. The present version analyzes the iteration
itself, bounds the error in mean square, and removes the logarithm.

Theorem~\ref{thm:main} addresses the gap between these results: it
controls the mean-square error of SGD at a data-driven stopping index,
uniformly over a source class and at a constant step size for which
sampling fluctuations need not be negligible.

\paragraph{Proof strategy.}
The proof inherits its structure from the continuous-time analysis of the
earlier version \cite{jahn2026sde}: the decomposition into the mean
iteration and its fluctuations, a Volterra inequality for the loss that
the incoherence closes, and a block contraction that localizes the
stopping index. Here these steps are carried out directly for the
discrete iteration. The bound in mean square and the removal of the
logarithmic factor of the high-probability bound in \cite{jahn2026sde}
rest on two new ingredients, described below: a spectral split adapted to
the block in which the rule stops, and a supermartingale argument beyond
a deterministic horizon.

The main difficulty is the multiplicative sampling noise. Writing the
stochastic gradient as its mean plus a fluctuation,
$\nabla F_{i_k}(\theta_k)=\frac1nX^\top r_k+X^\top R_ke_{i_k}$, where
$R_k:=\diag(r_k)-\frac1nr_k\mathbf1^\top$ and $\mathbf1$ is the all-ones
vector. This noise has zero conditional mean and a covariance that depends
on the current residual and vanishes with it. This distinguishes the
noise of SGD from an additive isotropic one, keeps the dynamics far from a
Langevin diffusion \cite{schertzer2024stochastic}, and lets constant-step
SGD converge without averaging when the system is consistent
\cite{varre2021last}. For the stopping problem it means that the residual
drives both the stopping rule and the amplitude of the sampling
fluctuations, which at a constant macroscopic step size, one that does not
shrink with $n$, cannot be treated as a small perturbation of the mean
iteration. Bounds at a fixed iteration
must therefore be supplemented by control of the random stopping index.

Applied to second moments, discrete Duhamel formulas and the
orthogonality of martingale increments show that the expected loss
$\E\norm{r_k}^2$ satisfies a closed discrete Volterra inequality: it is
bounded by a deterministic spectral term plus its own convolution with
the kernel $h_\ell=\gamma^2\sum_{j\le m}\mu_j^2\lambda_j^2(1-2\gamma\lambda_j)^{\ell-1}$,
as established in Proposition~\ref{prop:loss}. An exact cancellation in the residual second
moments makes the mass of this kernel equal to the noise-feedback strength
$\nu$ in \eqref{eq:incoherence}. The step-size condition $\nu\le1/4$ makes the
feedback a contraction, and $\nu=1$ is the ceiling beyond which the
argument no longer controls the second moment. The centered fluctuations
satisfy a second Volterra inequality, with kernel mass at most $4\nu/3$,
whose resolvent yields the variance bound needed to initialize the
localization of the stopping index. Volterra equations of this
convolution type describe the loss of SGD on least squares exactly in the
high-dimensional limit \cite{paquette2021sgd}, where the same kernel mass
separates bounded from divergent dynamics, and they are inherited by a
diffusion approximation of SGD under a delocalization condition on the
singular vectors of the data \cite{paquette2025homogenization}. Our
analysis is a finite-$n$, non-asymptotic counterpart for a fixed
ill-conditioned design, in which the incoherence quantifies that
delocalization and closes the inequality.

The argument has four steps. First, the mean SGD iterate follows
Landweber iteration, whose spectral filters identify the oracle scale.
Second, martingale orthogonality and the two Volterra inequalities
control the accumulated loss and the variance at that scale.
Third, on paths that have not yet stopped, the residual energy
contracts over successive oracle-length blocks. The weakly damped
spectral tail contributes a remainder that vanishes with $\delta/\rho$.
Finally, a spectral split converts residual control into parameter
control: the discrepancy inequality handles the large eigenvalues,
while the source condition and a martingale maximal inequality handle
the small ones.

In \cite{jahn2026sde}, the stopping index is controlled up to a horizon
of logarithmically many oracle times, and a single split at that horizon
costs the logarithmic factor. Here the split is placed at the end of the
block in which the rule stops, so its cost grows linearly in the block
index, while the probability of surviving a block decays geometrically.
Beyond a deterministic horizon of logarithmically many blocks, the
squared error stopped at $\tau$ is a supermartingale, so the error at
$\tau$ is controlled by the error at the horizon, where the surviving
residual energy is already of the order of a positive power of
$\delta/\rho$.
The drift inequality behind this supermartingale also bounds the expected
number of iterations beyond the horizon. No lower bound on the stopping
index is required.

\paragraph{Organization.}
Section~\ref{sec:analysis} proves Theorem~\ref{thm:main} through estimates
for the mean iteration and fluctuations, localization and termination of
the stopping rule, and control of the reconstruction error. Section~\ref{sec:numerics}
examines numerically how large the safety factor has to be for a given
step size, comparing the sufficient condition of Theorem~\ref{thm:main}
with the stopping behavior observed within the simulation budget.
Appendix~\ref{sec:proofs} collects the remaining proofs and the explicit
constants.

\section{Analysis}\label{sec:analysis}

We work under Assumption~\ref{ass:main}, with the iteration
\eqref{eq:sgd} and stopping rule \eqref{eq:dp} introduced above.
The same observation vector is used at every iteration.
For a scalar, $\norm{\cdot}$ denotes absolute value.
We write $\mathcal R(A)$ and $\mathcal N(A)$
for the range and null space of a matrix, $e_i$ for the $i$th canonical
vector of $\R^n$, and $\mathbf1$ for the all-ones vector.
The symbols $a\vee b$ and $a\wedge b$ denote maximum and minimum,
and $\zeta$ denotes the Riemann zeta function. Empty sums equal zero.

The row indices $(i_k)_{k\ge0}$ in \eqref{eq:sgd} are independent and
uniform on $\{1,\ldots,n\}$, defined on a probability space
$(\Omega,\mathcal F,\Prob)$. Set $\mathcal F_0=\{\varnothing,\Omega\}$ and
$\mathcal F_k=\sigma(i_0,\ldots,i_{k-1})$.
In particular $i_k$ is independent of $\mathcal F_k$. The error
$\eta_k=\theta_k-\theta^\star$ and the residual
$r_k=X\theta_k-y=X\eta_k-\eps$ evolve by \eqref{eq:residual}, and the
residual is checked at every iteration in \eqref{eq:dp}.
For a fixed $k$, the iterate and residual are $\mathcal F_k$-measurable.
Moreover
\[
 \{\tau\le k\}=\bigcup_{a=0}^k
       \{\norm{r_a}\le\kstop\sqrt n\,\delta\}\in\mathcal F_k,
\]
so $\tau$ is a stopping time. The iterates in all moment estimates are
the unstopped iterates, defined by \eqref{eq:sgd} for every $k$ even
after the stopping rule has first been satisfied.

Up to a deterministic horizon $H$, only the $n^H$ row sequences
$(i_0,\ldots,i_{H-1})$ occur, each with probability $n^{-H}$. All
iterates up to $H$ are therefore bounded, all moments below are finite,
and conditional expectations given $\mathcal F_k$ are defined on each of
its $n^k$ atoms, so conditional identities and inequalities hold
pointwise. 

We use the exponents
\begin{equation}\label{eq:exponents}
 d_\beta:=1-\frac1\beta,\qquad
 \alpha:=d_\beta-2r,\qquad
 c_G:=\frac\alpha{2r+1}.
\end{equation}
By \ref{ass:compatibility}, $\alpha>0$; it measures the gap in the
compatibility condition.

The source condition and the fact that every increment belongs to
$\mathcal R(X^\top)$ imply $\eta_k\in\mathcal R(X^\top)$ for every $k$.
Directions in $\mathcal N(X^\top)$ satisfy
$\langle r_k,u_j\rangle=-\langle\eps,u_j\rangle$ for $j>m$.
They are retained in all residual tail bounds and have no fluctuations.
Their total squared norm is at most $n\delta^2$. In particular the analysis
allows $m<n$ and observations outside $\mathcal R(X)$.

In the fixed basis, put
\[
 w_j=\langle v,v_j\rangle\quad(j\le m),\qquad
 w_j=0\quad(j>m),\qquad \eps_j=\langle\eps,u_j\rangle.
\]
Then $\sum_{j\le m}w_j^2\le\rho^2$ and
$\langle r_0,u_j\rangle=\sqrt n\,\lambda_j^{r+1/2}w_j-\eps_j$
for every $j\le n$. Define globally
$\bar r_k=\E r_k$, $\widetilde r_k=r_k-\bar r_k$,
$\bar\eta_k=\E\eta_k$, and
$\widetilde\eta_k=\eta_k-\bar\eta_k$.
By linearity,
$\widetilde r_k=X\widetilde\eta_k$.

The step-size assumption also implies the useful pointwise stability bounds
\begin{equation}\label{eq:step-consequences}
 \gamma\lambda_1\le2\nu\le\tfrac12,
 \qquad
 \gamma R_x^2\le2\nu\le\tfrac12,
 \qquad R_x^2:=\max_i\norm{x_i}^2,
\end{equation}
because $\mu_j^2\ge1$ and
$\norm{x_i}^2=n\sum_{j\le m}\lambda_j u_{j,i}^2
\le\sum_{j\le m}\mu_j^2\lambda_j$.
In addition $\kappa\le\bar\kappa:=C_\lambda\zeta(\beta)$.
The quantity $\mu_j^2$ is defined only for $j\le m$; no incoherence
parameter is used for the zero modes.

If the eigenvectors $u_j$ have entries of constant modulus $n^{-1/2}$, as for the
normalized Hadamard basis used in Section~\ref{sec:numerics}, then
$\mu_j^2=1$ for all $j$ and $\nu=\gamma\kappa/2$; for a circulant $K$,
the real trigonometric eigenbasis gives $\mu_j^2\le2$. For designs with
independent standard Gaussian entries, $\mu_\star^2\le\mu^2\le C\log n$
with probability at least $1-n^{-c}$, for absolute constants $C,c>0$
\cite{jiang2005maxima}. On the
designs examined in \cite[Section~3.1]{jahn2026sde}, a discretized
integral operator and two data-driven designs, $\mu_\star^2$ grows by less
than a factor of two over $n\in[128,2048]$, whereas $\mu^2$ exceeds $10^3$
on the data-driven ones.

\begin{remark}[On the step-size condition \ref{ass:step}]\label{rem:stepsize}
For online SGD on least squares, the step size is classically restricted
through a fourth-moment condition on the features,
$\E[\norm x^2xx^\top]\preceq R^2\Sigma$ with $\gamma R^2$ below an absolute
constant \cite{bach2013nonstrongly,dieuleveut2017harder,jain2018parallelizing};
it holds with $R^2=c_4\Tr\Sigma$ when the kurtosis of the features is at
most $c_4$ in every direction \cite{dieuleveut2017harder}. The condition
bounds the energy that the gradient noise injects along each eigendirection
of $\Sigma$ by the current risk, which closes the second-moment recursion.
Assumption~\ref{ass:step}, $\gamma\mu_\star^2\Tr\Sigma\le\frac12$, has the
same form, with the spectrum-weighted incoherence in place of the kurtosis
constant, and the incoherence enters the proof of
Proposition~\ref{prop:loss} at exactly this point. For a fixed design,
$\mu_j^2$ bounds the empirical kurtosis of the rows along the $j$th
singular direction, but as a maximum over the rows rather than an average.
This is what deterministic noise requires. The injected energy weights the
rows by the squared entries of the residual. In the random-design setting
these weights are quadratic in the features, apart from the observation
noise, which is controlled separately, so fourth moments suffice. A
deterministic noise vector, in contrast, may concentrate on the rows where
a singular vector is largest, and its component in $\mathcal N(X^\top)$ is
never fitted; only the maximum over the rows bounds every such weighting.
On the designs examined in \cite[Section~3.1]{jahn2026sde}, the step
admitted by \ref{ass:step} exceeds the capacity-based step size of
\cite{varre2021last} by one to two orders of magnitude. Under white noise
a kurtosis-type condition might suffice, but a different difficulty
arises: the discrepancy principle then has to be calibrated with a safety
factor $1+O(n^{-1/2})$ \cite{blanchard2018optimal}, which for a macroscopic
step lies below the level to which the sampling noise inflates the
residual (Remark~\ref{rem:threshold}).
\end{remark}

\subsection{Mean iteration}\label{sec:mean-flow}

Since $\E[e_{i_k}e_{i_k}^\top\mid\mathcal F_k]=I_n/n$, taking
expectations in \eqref{eq:residual}, and likewise in
$\eta_{k+1}=\eta_k-\gamma x_{i_k}(r_k)_{i_k}$, shows that the mean
iterate is the Landweber iterate with step $\gamma$:
\begin{equation}\label{eq:mean}
 \bar r_k=(I-\gamma K)^kr_0,\qquad
 \bar\eta_k=(I-\gamma\Sigma)^k\eta_0+\phi_k(\Sigma)\frac{X^\top\eps}n,
\end{equation}
where $\phi_k(\lambda)=(1-(1-\gamma\lambda)^k)/\lambda$ for $\lambda>0$
and $\phi_k(0)=k\gamma$. In the following we need a few bounds for this
iteration.

\begin{lemma}[Bounds for the mean iteration]\label{lem:mean}
Assume \ref{ass:spectrum}, \ref{ass:source} and $\gamma\lambda_1\le1/2$.
The mean loss $f_k:=\norm{\bar r_k}^2$ is nonincreasing in $k$, and
\begin{align}
 \norm{(I-\gamma K)^kX\eta_0}
 &\le\sqrt n\,c_r\rho(k\gamma)^{-(r+1/2)},\quad k\ge1,
 &c_r&=\Bigl(\frac{r+1/2}e\Bigr)^{r+1/2},\label{eq:mean-residual}\\
 f_k&\le2n\bigl[B_r\rho^2(1\vee k\gamma)^{-(2r+1)}+\delta^2\bigr],
 &B_r&=\max\{C_\lambda^{2r+1},c_r^2\},\label{eq:mean-envelope}\\
 \gamma\sum_{k=0}^{H-1}f_k
 &\le n\Bigl(\frac43C_\lambda^{2r}\rho^2+2\delta^2H\gamma\Bigr),
 &&H\in\mathbb N_0.\label{eq:mean-cumulative}
\end{align}
Moreover, for every integer $k\ge0$ and $x\in[0,1]$,
\begin{equation}\label{eq:filter}
 (1-(1-x)^k)^2\le kx .
\end{equation}
\end{lemma}
\begin{proof}
The eigenvalues of $I-\gamma K$ lie in $[1/2,1]$, so
$f_k=\sum_j(1-\gamma\lambda_j)^{2k}\langle r_0,u_j\rangle^2$ is
nonincreasing. By \ref{ass:source},
$\norm{(I-\gamma K)^kX\eta_0}^2
=n\sum_{j\le m}\lambda_j^{2r+1}(1-\gamma\lambda_j)^{2k}w_j^2$.
With $(1-\gamma\lambda)^{2k}\le e^{-2k\gamma\lambda}$ and
$\sup_{\lambda>0}\lambda^ae^{-t\lambda}=(a/(et))^a$ for $a=2r+1$ and
$t=2k\gamma$, this gives \eqref{eq:mean-residual}. For $k\gamma\le1$,
the bound $\lambda_j\le C_\lambda$ gives instead
$\norm{(I-\gamma K)^kX\eta_0}^2\le nC_\lambda^{2r+1}\rho^2$. Since
$\bar r_k=(I-\gamma K)^kX\eta_0-(I-\gamma K)^k\eps$ and
$\norm{(I-\gamma K)^k\eps}\le\sqrt n\,\delta$, the inequality
$\norm{a-b}^2\le2\norm a^2+2\norm b^2$ yields \eqref{eq:mean-envelope}.
In the same splitting, the identity
$\gamma\sum_{k\ge0}\lambda(1-\gamma\lambda)^{2k}=1/(2-\gamma\lambda)\le2/3$
and $\lambda_j^{2r}\le C_\lambda^{2r}$ bound the summed approximation
part by $\frac43nC_\lambda^{2r}\rho^2$, which gives
\eqref{eq:mean-cumulative}. Finally, Bernoulli's inequality gives
$0\le1-(1-x)^k\le\min\{1,kx\}$, hence \eqref{eq:filter}.
\end{proof}

\subsection{Fluctuations: a Volterra second-moment inequality}
\label{sec:fluctuations}\label{sec:fluct}

The sampling fluctuations are centered conditionally on the past, but
their variance depends on the current residual. We first identify this
dependence, then derive the two moment inequalities needed later:
one for the total residual energy and one for its centered part.
Lemmas~\ref{lem:integrated-loss} and~\ref{lem:resolvent-decay} and
Proposition~\ref{prop:variance} are proved in
Appendix~\ref{sec:proofs-fluctuations}; the other proofs are given in place.

\paragraph{Centered innovations and their covariance.}

Set $r_{j,k}=\langle r_k,u_j\rangle$, $\Lc_k=\E\norm{r_k}^2$, and
\[
 R_k=\diag(r_k)-\frac1n r_k\mathbf1^\top,\qquad
 \Delta_{k+1}=-\gamma n K R_k e_{i_k}.
\]
Then
\begin{equation}\label{eq:innovation}
 r_{k+1}=(I-\gamma K)r_k+\Delta_{k+1},\qquad
 \E[\Delta_{k+1}\mid\mathcal F_k]=0.
\end{equation}
Indeed $R_k\mathbf1=0$ and
$-\gamma nKR_ke_{i_k}=-\gamma nKe_{i_k}(r_k)_{i_k}+\gamma Kr_k$.
The full conditional covariance is
\begin{equation}\label{eq:full-covariance}
 \E[\Delta_{k+1}\Delta_{k+1}^\top\mid\mathcal F_k]
 =\gamma^2nKR_kR_k^\top K
 =\gamma^2\left[nK\diag(r_k^2)K-Kr_kr_k^\top K\right].
\end{equation}
Here $r_k^2$ denotes the vector of squared coordinates.
For $j\le n$, put $\Delta_{j,k+1}=\langle\Delta_{k+1},u_j\rangle$.
For $j\le m$, projecting \eqref{eq:full-covariance} gives
\begin{equation}\label{eq:innovation-variance}
 \E[\Delta_{j,k+1}^2\mid\mathcal F_k]
 =\gamma^2\lambda_j^2
 \left(n\sum_i u_{j,i}^2(r_k)_i^2-r_{j,k}^2\right)
 \le\gamma^2\mu_j^2\lambda_j^2\norm{r_k}^2.
\end{equation}
For $j>m$ one has $\Delta_{j,k+1}=0$ identically, since $u_j^\top K=0$.
The negative term in \eqref{eq:innovation-variance} is essential for
the uncentered second-moment comparison below.

Iterating the centered recursion gives a discrete Duhamel formula.
At a deterministic horizon, cross terms vanish by conditional
centering. The resulting identity is the discrete counterpart of
the It\^o isometry.

\begin{lemma}[Duhamel representation and orthogonality]\label{lem:duhamel}
Let $a_j=1-\gamma\lambda_j$. For every deterministic $k\ge0$ and
$j\le m$,
\begin{equation}\label{eq:duhamel}
 \widetilde r_{j,k}
 =\sum_{a=0}^{k-1}a_j^{k-1-a}\Delta_{j,a+1},\qquad
 \E\widetilde r_{j,k}^2
 =\sum_{a=0}^{k-1}a_j^{2(k-1-a)}\E\Delta_{j,a+1}^2.
\end{equation}
The fluctuations vanish for $j>m$. In particular
\begin{equation}\label{eq:coordinate-fluctuation}
 \E\widetilde r_{j,k}^2
 \le\gamma^2\mu_j^2\lambda_j^2
       \sum_{a=0}^{k-1}a_j^{2(k-1-a)}\Lc_a.
\end{equation}
\end{lemma}
\begin{proof}
Subtract the deterministic mean recurrence from \eqref{eq:innovation}
and iterate from $\widetilde r_{j,0}=0$.
For integers $a<b$, $\Delta_{j,a+1}$ is $\mathcal F_b$-measurable,
so the tower property yields
\[
 \E[\Delta_{j,a+1}\Delta_{j,b+1}]
 =\E\!\left[\Delta_{j,a+1}
       \E[\Delta_{j,b+1}\mid\mathcal F_b]\right]=0.
\]
The multipliers at a deterministic $k$ are deterministic.
Expansion of the finite square thus proves the second equality in
\eqref{eq:duhamel}. Apply \eqref{eq:innovation-variance} for the bound.
The assertion for zero modes follows directly from \eqref{eq:innovation}.
\end{proof}

\paragraph{Residual second moments and cumulative loss.}
For the total residual energy, the negative term in the conditional
covariance combines with the squared mean multiplier. The resulting
factor is $1-2\gamma\lambda_j$, whose geometric sum gives exactly
the feedback mass $\nu$.

\begin{proposition}[Volterra inequality and conditional restart]\label{prop:loss}
Under $\nu\le1/4$, let $Q=I-2\gamma K$ and
\begin{equation}\label{eq:kernel}
 \varphi_k=\norm{Q^{k/2}r_0}^2,\qquad
 h_\ell=\gamma^2\sum_{j\le m}\mu_j^2\lambda_j^2
                 (1-2\gamma\lambda_j)^{\ell-1},\quad \ell\ge1.
\end{equation}
The matrix $Q$ is positive semidefinite, and
\begin{equation}\label{eq:volterra}
 \Lc_k\le \varphi_k+\sum_{a=0}^{k-1}h_{k-a}\Lc_a,
 \qquad \sum_{\ell\ge1}h_\ell=\nu,
 \qquad \sup_{k\ge0}\Lc_k\le\frac{\norm{r_0}^2}{1-\nu}.
\end{equation}
For any deterministic integer $a\ge0$, the same Volterra argument,
conditionally on $\mathcal F_a$, gives
\begin{equation}\label{eq:conditional-loss}
 \E[\norm{r_{a+b}}^2\mid\mathcal F_a]
 \le\norm{Q^{b/2}r_a}^2+\frac\nu{1-\nu}\norm{r_a}^2,
 \qquad b\in\mathbb N_0.
\end{equation}
\end{proposition}
\begin{proof}
Retain the negative term in \eqref{eq:innovation-variance}. Combining it
with the square of the mean multiplier gives the exact identity
\begin{equation}\label{eq:exact-second-moment}
 \E[r_{j,k+1}^2\mid\mathcal F_k]
 =(1-2\gamma\lambda_j)r_{j,k}^2
  +\gamma^2n\lambda_j^2\sum_i u_{j,i}^2(r_k)_i^2.
\end{equation}
Here $(1-\gamma\lambda_j)^2-\gamma^2\lambda_j^2
=1-2\gamma\lambda_j$. Since $\gamma\lambda_1\le1/2$, all multipliers
$1-2\gamma\lambda_j$ lie in $[0,1]$. For $j\le m$, taking expectations
and using $n\,u_{j,i}^2\le\mu_j^2$ yields
\[
 \E r_{j,k+1}^2\le(1-2\gamma\lambda_j)\E r_{j,k}^2
                  +\gamma^2\mu_j^2\lambda_j^2\Lc_k.
\]
Unrolling and summing over $j$, including the frozen zero modes, proves
the first inequality of \eqref{eq:volterra}. The geometric series gives
\[
 \sum_{\ell\ge1}h_\ell
 =\gamma^2\sum_{j\le m}
       \frac{\mu_j^2\lambda_j^2}{2\gamma\lambda_j}=\nu.
\]
For a finite integer $H$, let $S_H=\max_{0\le k\le H}\Lc_k$.
All terms are nonnegative and $\varphi_k\le\norm{r_0}^2$, so
$\Lc_k\le\norm{r_0}^2+\nu S_H$ for every $k\le H$.
Taking the maximum and subtracting $\nu S_H$ gives
$S_H\le\norm{r_0}^2/(1-\nu)$.
This bound holds for every $H$, proving the uniform loss bound.

For the conditional statement, fix a deterministic $a$. The indices
$i_a,i_{a+1},\ldots$ are independent of $\mathcal F_a$, so on each atom
of $\mathcal F_a$ the iterates from $a$ on are SGD started at $\theta_a$,
with residual $r_a$. The argument above, applied to this chain, gives for
$\Lc_\ell^{(a)}=\E[\norm{r_{a+\ell}}^2\mid\mathcal F_a]$
\[
 \Lc_\ell^{(a)}
 \le\norm{Q^{\ell/2}r_a}^2+
            \sum_{z=0}^{\ell-1}h_{\ell-z}\Lc_z^{(a)},
 \qquad
 \max_{\ell\le b}\Lc_\ell^{(a)}\le\frac{\norm{r_a}^2}{1-\nu}.
\]
Substitute the second bound into the convolution term at $\ell=b$.
\end{proof}

Summing the Volterra inequality controls the total residual energy
that feeds the sampling noise. This estimate will enter both the
stochastic residual tail and the maximal parameter-tail bound.

\begin{lemma}[Integrated loss]\label{lem:integrated-loss}
Under \ref{ass:spectrum}, \ref{ass:source}, and $\nu\le1/4$, for every integer $H\ge0$,
\begin{equation}\label{eq:cumulative-loss}
 \gamma\sum_{k=0}^{H-1}\Lc_k
 \le\frac{n}{1-\nu}
       \left(C_\lambda^{2r}\rho^2+2\delta^2 H\gamma\right).
\end{equation}
\end{lemma}

\begin{remark}[The comparison term and the mean loss]\label{rem:kernel}
The forcing $\varphi_k=\norm{(I-2\gamma K)^{k/2}r_0}^2$ in
\eqref{eq:volterra} differs from the squared mean
$f_k=\norm{(I-\gamma K)^kr_0}^2$.
The former exploits the cancellation in
\eqref{eq:exact-second-moment}; the latter is needed to isolate the
actual centered variance. This leads to the second kernel below.
\end{remark}

\paragraph{Centered fluctuations and the Volterra resolvent.}
At the first checkpoint the mean residual must lie strictly below the
threshold, and its fluctuations must carry the factor $\gamma\mu^2$.
We obtain the latter from a variance envelope. Its proof uses the spectral sum estimates of
Lemma~\ref{lem:spectral-tail}, stated in Appendix~\ref{sec:proofs}.

Let $V_k=\E\norm{\widetilde r_k}^2$, so $\Lc_k=f_k+V_k$.
For sequences indexed by $\mathbb N_0$, write
$(a*b)_k=\sum_{z=0}^k a_{k-z}b_z$.
Set $\widehat h_0=0$ and, for $\ell\ge1$,
\begin{equation}\label{eq:centered-kernel}
 \widehat h_\ell=\gamma^2\sum_{j\le m}\mu_j^2\lambda_j^2
                   (1-\gamma\lambda_j)^{2(\ell-1)},\qquad
 \widehat\nu=\gamma\sum_{j\le m}
                   \frac{\mu_j^2\lambda_j}{2-\gamma\lambda_j}.
\end{equation}
The squared Landweber factors produce the mass $\widehat\nu$.
Since $\gamma\lambda_j\le1/2$, this mass remains uniformly below one,
and the centered inequality can be inverted by a nonnegative resolvent.

\begin{proposition}[Centered Volterra inequality]\label{prop:centered}
Under $\nu\le1/4$,
\begin{equation}\label{eq:centered-volterra}
 V_k\le(\widehat h*\Lc)_k,\qquad
 \Lc_k\le f_k+(\widehat h*\Lc)_k,\qquad
 \norm{\widehat h}_{\ell^1}=\widehat\nu
              \le\frac43\nu\le\frac13.
\end{equation}
The nonnegative resolvent
$\widehat{\mathcal Q}=\sum_{\ell\ge1}\widehat h^{*\ell}$ satisfies
\begin{equation}\label{eq:centered-resolvent}
 \norm{\widehat{\mathcal Q}}_{\ell^1}
       =\frac{\widehat\nu}{1-\widehat\nu},\qquad
 \Lc_k\le f_k+(\widehat{\mathcal Q}*f)_k,\qquad
 V_k\le(\widehat{\mathcal Q}*f)_k.
\end{equation}
Here $\widehat{\mathcal Q}_0=0$.
\end{proposition}
\begin{proof}
Sum \eqref{eq:coordinate-fluctuation} over $j\le m$ and use the
vanishing fluctuations in zero modes to obtain the first inequality.
The identity $\Lc_k=f_k+V_k$ gives the second.
Geometric summation and $\gamma\lambda_j\le1/2$ give
\[
 \sum_{\ell\ge1}\widehat h_\ell
 =\gamma^2\sum_{j\le m}
     \frac{\mu_j^2\lambda_j^2}{1-(1-\gamma\lambda_j)^2}
 =\widehat\nu\le\frac23\gamma\sum_j\mu_j^2\lambda_j=\frac43\nu.
\]
For nonnegative sequences
$\norm{a*b}_{\ell^1}=\norm a_{\ell^1}\norm b_{\ell^1}$, so
$\norm{\widehat h^{*\ell}}_{\ell^1}=\widehat\nu^\ell$.
Thus the resolvent converges in $\ell^1$ with the stated norm.
There is also no limiting issue at a fixed index $k$:
$\widehat h^{*\ell}$ is supported on indices at least $\ell$.
After $k+1$ iterations of the Volterra inequality, its convolution
remainder vanishes at $k$, proving the loss bound.
Insert that bound into $V\le\widehat h*\Lc$ and use
$\widehat h+\widehat h*\widehat{\mathcal Q}=\widehat{\mathcal Q}$
to prove the last bound.
\end{proof}

Proposition~\ref{prop:centered} bounds the variance by the mean loss. With the $\ell^1$ norm of
$\widehat{\mathcal Q}$ alone, this gives only
$V_k\le\frac{\widehat\nu}{1-\widehat\nu}f_0$, which does not decay
in $k$. The decay of the variance requires a pointwise bound on the
resolvent. 

\begin{lemma}[Decay of the discrete resolvent]\label{lem:resolvent-decay}
Put $p_\beta=2-1/\beta=1+d_\beta$.
Under \ref{ass:spectrum} and $\nu\le1/4$, there is a constant
$D_{\mathcal Q}$ depending only on $\beta,C_\lambda$ such that,
for every integer $k\ge1$,
\begin{equation}\label{eq:kernel-decay}
 \widehat h_k\le e\,c_{2,\beta}\gamma^2\mu^2
                         (1\vee k\gamma)^{-p_\beta},\qquad
 \widehat{\mathcal Q}_k\le D_{\mathcal Q}\gamma^2\mu^2
                         (1\vee k\gamma)^{-p_\beta}.
\end{equation}
Here $c_{2,\beta}$ is defined in \eqref{eq:spectral-constants};
an explicit $D_{\mathcal Q}$ is given in \eqref{eq:resolvent-constant}.
\end{lemma}

By the compatibility condition \ref{ass:compatibility},
$p_\beta=1+d_\beta>2r+1$, so the resolvent decays faster than the
mean-loss envelope \eqref{eq:mean-envelope}. The convolution
$\widehat{\mathcal Q}*f$ therefore decays like the mean loss, and the
pointwise bound of Lemma~\ref{lem:resolvent-decay} contributes the
factor $\gamma\mu^2$.

\begin{proposition}[Variance envelope]\label{prop:variance}
Under Assumption~\ref{ass:main}, for every $k\ge0$,
\begin{equation}\label{eq:variance-envelope}
 V_k\le C_V\gamma\mu^2 n
       \left[\rho^2(1\vee k\gamma)^{-(2r+1)}+\delta^2\right],
\end{equation}
where $C_V$ depends only on $r,\beta,C_\lambda$.
An explicit choice is given in \eqref{eq:variance-constant}.
\end{proposition}

At an oracle-scale index, the envelope gives a contribution of
order $\gamma\mu^2n\delta^2$.
For fixed $\gamma\mu^2$, a second-moment estimate at this single
checkpoint does not yield a failure probability tending to zero with
$\delta/\rho$. To obtain such a probability, we next exploit the
contraction of residual energy on paths that have not yet stopped.

\begin{remark}[The safety factor]\label{rem:threshold}
The condition $\kstop>1/(1-\nu)$ of Theorem~\ref{thm:main} comes from the
drift inequality of Proposition~\ref{prop:termination}, which makes the
squared error stopped at $\tau$ a supermartingale; the block contraction
of Subsection~\ref{sec:localization} alone needs only the weaker condition
$\nu/(1-\nu)+\kstop^{-2}<1$. Neither threshold is expected to be sharp.
As $\kstop$ approaches $1/(1-\nu)$ at fixed $\nu>0$, only the constant
$\widetilde C_\tau$ of Theorem~\ref{thm:main} blows up, like
$g_{\mathrm d}^{-1}$; the growth as $\kstop\downarrow1$ is discussed
after the theorem. Some excess
over one is natural for constant step sizes, however. Near the oracle
time, the residual of Landweber iteration consists mostly of noise that
has not yet been fitted, and the sampling noise inflates the expected
squared residual of SGD above this level, by a factor of up to about
$1/(1-\nu)$, as the Volterra inequality of Proposition~\ref{prop:loss}
indicates. Below the inflated level, the rule can stop only through
downward fluctuations of the residual, which second-moment arguments do
not capture. This applies to any deterministic noise placed in directions
that are fitted only after the oracle time. Noise in $\mathcal N(X^\top)$
is the extreme case: it is never fitted, and the inflated level becomes a
stationary floor. Section~\ref{sec:numerics} examines this case numerically,
including when downward fluctuations suffice to stop the rule below the
floor.
\end{remark}

\subsection{Stability: no late stopping}\label{sec:localization}\label{sec:thm1}

At an oracle-scale checkpoint, the preceding estimates place
the expected residual energy at the noise scale. Over each subsequent
block, modes with $\lambda_j\ge t_b^{-1}$ contract by a fixed factor.
The remaining modes form a spectral tail whose deterministic component
is at the noise scale and whose stochastic component is smaller.
On the event that the rule has not yet stopped, these estimates yield
a contraction of the expected residual energy restricted to that event.
Lemmas~\ref{lem:horizon} and~\ref{lem:block}--\ref{lem:tail-subcritical}
are proved in Appendix~\ref{sec:proofs-stability}; the other proofs are
given in place.
Throughout this subsection and the next, assume the hypotheses of
Theorem~\ref{thm:main}, with $g_{\mathrm d}$ as there. The contraction
is governed by the gap
\begin{equation}\label{eq:contraction-constants}
 g_{\mathrm c}=1-\frac{\nu}{1-\nu}-\frac1{\kstop^2},
\end{equation}
which is positive by Lemma~\ref{lem:horizon} below. Write
$s:=\delta/\rho\le e^{-1}$ for the noise-to-signal ratio, and set
\begin{equation}\label{eq:parameters}
 \vartheta=\min\Bigl\{1,\frac{g_{\mathrm c}\kstop^2}{12}\Bigr\},\qquad
 J=\Bigl\lceil\frac12\log\frac4{g_{\mathrm c}}\Bigr\rceil,\qquad
 C_b=\max\Bigl\{\Bigl(\frac3\vartheta\Bigr)^{\frac2{2r+1}},
           \Bigl(\frac{2c_r}{\kstop-1}\Bigr)^{\frac2{2r+1}}\Bigr\},
\end{equation}
with $c_r$ from \eqref{eq:mean-residual}, and
\begin{equation}\label{eq:horizon}
 \begin{gathered}
 \omega_b=\bigl(1+C_b^{-(r+1/2)}\bigr)^2,\qquad
 q_\star=e^{-2J}+\frac\nu{1-\nu}+\frac{(1+\vartheta)\,\omega_b}{\kstop^2},\\
 b=\left\lceil\frac{C_b\tprior}{\gamma}\right\rceil,\quad
 t_b=\gamma b,\quad
 B=J b,\quad
 N_c=\left\lceil\frac{c_G\log(1/s)}{|\log q_\star|}\right\rceil,\quad
 N=(N_c+1)B,\quad T=N\gamma.
 \end{gathered}
\end{equation}
The spectral split is placed at $t_b^{-1}$, with $t_b$ of the order of
the oracle time. Checkpoints are spaced by $B=J b$ steps, so that
the modes above the split contract by the factor $e^{-2J}$ between
checkpoints, and the horizon $N$ covers $N_c+1$ blocks. These quantities
are used only in the analysis. All checkpoints below are deterministic
integer indices.

\begin{lemma}[Horizon and uniform contraction constants]\label{lem:horizon}
Set
\begin{equation}\label{eq:uniform-constants}
 \bar q_\star=1-\frac{g_{\mathrm c}}2,\qquad
 \bar D=\frac{c_G}{|\log\bar q_\star|}+2.
\end{equation}
Then $0<g_{\mathrm c}<1$, $g_{\mathrm d}>0$, $J\ge1$, $C_b\ge1$, and
\begin{align}
 0<q_\star&\le\bar q_\star<1,\label{eq:uniform-contraction}\\
 C_b\tprior&\le t_b<C_b\tprior+\gamma
                         \le(C_b+1)\tprior,\label{eq:grid-rounding}\\
 q_\star^{N_c}&\le s^{c_G},\qquad
 N_c+1\le\bar D\log(1/s),\qquad
 1\le T/\tprior\le J(C_b+1)\bar D\log(1/s).
 \label{eq:horizon-bounds}
\end{align}
\end{lemma}

\paragraph{Initialization at the oracle scale.}
The length $b$ makes the mean residual strictly smaller than
the discrepancy threshold. The variance envelope then bounds the
fluctuations around it at every later index, in particular at the first
checkpoint.

\begin{lemma}[Base case]\label{lem:base}
For every integer $a\ge b$,
\begin{equation}\label{eq:base-case}
 \norm{\bar r_a}\le\frac{\kstop+1}{2}\sqrt n\,\delta,\qquad
 V_a\le B_0\gamma\mu^2n\delta^2,\qquad B_0=2C_V,
\end{equation}
where $C_V$ is the constant in Proposition~\ref{prop:variance}.
\end{lemma}
\begin{proof}
The mean residual is nonincreasing in norm. The mean estimates,
$t_b\ge C_b\tprior$ and the choice of $C_b$ give
\[
 \norm{\bar r_a}\le\norm{\bar r_b}
 \le\sqrt n(\rho c_rt_b^{-(r+1/2)}+\delta)
 \le\sqrt n\,\delta(c_rC_b^{-(r+1/2)}+1)
 \le\frac{\kstop+1}{2}\sqrt n\,\delta.
\]
Since $a\gamma\ge t_b\ge\tprior\ge1$, the variance envelope gives
\[
 V_a\le C_V\gamma\mu^2n
       (\rho^2(a\gamma)^{-(2r+1)}+\delta^2)
 \le2C_V\gamma\mu^2n\delta^2.
\]
\end{proof}

\paragraph{Block contraction and spectral tails.}
Place checkpoints at
\[
 a_\ell=(\ell+1)B,\qquad 0\le\ell\le N_c,
\]
so $a_0=B\ge b$ and $a_{N_c}=N$. Use the fixed spectral partition
\[
 \mathcal H_b=\{j\le n:\lambda_j\ge t_b^{-1}\},\qquad
 \mathcal T_b=\{j\le n:\lambda_j<t_b^{-1}\}.
\]
For any integer $a\ge0$ define
$E_a^{\rm tail}=\sum_{j\in\mathcal T_b}r_{j,a}^2$,
$\bar E_a^{\rm tail}=\sum_{j\in\mathcal T_b}\bar r_{j,a}^2$ and
$\widetilde E_a^{\rm tail}=\sum_{j\in\mathcal T_b}\widetilde r_{j,a}^2$.
The inequality $(x+y)^2\le(1+\vartheta)x^2+(1+\vartheta^{-1})y^2$ gives
\begin{equation}\label{eq:tail-split}
 E_a^{\rm tail}\le(1+\vartheta)\bar E_a^{\rm tail}
   +(1+\vartheta^{-1})\widetilde E_a^{\rm tail}.
\end{equation}
Over one block, the modes in $\mathcal H_b$ contract by the factor
$e^{-2J}$, the sampling noise adds at most the fraction $\nu/(1-\nu)$ of
the current energy, and the tail is carried over.

\begin{lemma}[Contraction over one block]\label{lem:block}\label{lem:contraction}
For every deterministic integer $a\ge0$,
\begin{equation}\label{eq:block}
 \E[\norm{r_{a+B}}^2\mid\mathcal F_a]
 \le\Bigl(e^{-2J}+\frac\nu{1-\nu}\Bigr)\norm{r_a}^2+E_a^{\rm tail}.
\end{equation}
\end{lemma}

Since the tail does not contract, its two parts in \eqref{eq:tail-split}
are bounded separately. For the mean residual, the source condition makes
the tail of the signal smaller than the noise once $t_b$ is of the order
of the oracle time, so this part stays at the noise scale.

\begin{lemma}[Deterministic spectral tail]\label{lem:tail-det}
For every integer $a\ge0$,
\begin{equation}\label{eq:det-tail}
 \bar E_a^{\rm tail}\le\omega_b\,n\delta^2.
\end{equation}
\end{lemma}

The stochastic part collects the sampling noise injected into the slow
modes during all earlier blocks. These modes receive little noise, and by
the compatibility condition \ref{ass:compatibility} the accumulated energy
is smaller than the threshold energy $n\delta^2$ by a power of $s$.

\begin{lemma}[Subcritical stochastic tail]\label{lem:tail-subcritical}
There is a constant $C_G$ depending only on $r,\beta,C_\lambda$ and $g_{\mathrm c}$
such that, for every deterministic integer $a$ with $0\le a\le N$,
\begin{equation}\label{eq:stoch-tail}
 \E\widetilde E_a^{\rm tail}
 \le C_G\gamma\mu^2n\delta^2s^{c_G}.
\end{equation}
An explicit choice is given in \eqref{eq:tail-constant}.
\end{lemma}

\paragraph{Residual energy before stopping.}
The preceding tail bounds play different roles. The deterministic
tail is absorbed into the contraction because the residual exceeds
the discrepancy threshold on every surviving path. The stochastic
tail contributes the smaller additive term $s^{c_G}$.
To express this argument, for $0\le\ell\le N_c$ let
\[
 m_\ell=\frac{\E[\ind_{\{\tau>a_\ell\}}\norm{r_{a_\ell}}^2]}
                  {n\delta^2}.
\]
At the first checkpoint the mean residual is already below the
threshold, so only fluctuations can keep the rule from stopping there.
All survival energies are therefore bounded by multiples of $\gamma\mu^2$.
\begin{lemma}[Block recursion]\label{lem:recursion}
One has $m_0\le B_0'\gamma\mu^2$ with
$B_0'=\bigl(2\kstop/(\kstop-1)\bigr)^2B_0$ and, for $0\le\ell<N_c$,
\begin{equation}\label{eq:survival-recursion}
 m_{\ell+1}\le q_\star m_\ell+(1+\vartheta^{-1})C_G\gamma\mu^2s^{c_G}.
\end{equation}
\end{lemma}
\begin{proof}
For the initial bound, recall $a_0=B\ge b$. On $\{\tau>a_0\}$ one has
$\norm{r_{a_0}}>\kstop\sqrt n\,\delta$, while
$\norm{\bar r_{a_0}}\le\frac{\kstop+1}{2}\sqrt n\,\delta$ by
Lemma~\ref{lem:base}. Hence
$\norm{\widetilde r_{a_0}}\ge\norm{r_{a_0}}-\norm{\bar r_{a_0}}
>\frac{\kstop-1}{2\kstop}\norm{r_{a_0}}$ on this event, and
\[
 n\delta^2m_0
 \le\Bigl(\frac{2\kstop}{\kstop-1}\Bigr)^2\E\norm{\widetilde r_{a_0}}^2
 =\Bigl(\frac{2\kstop}{\kstop-1}\Bigr)^2V_{a_0}
 \le B_0'\gamma\mu^2n\delta^2 .
\]
For the recursion, since $\{\tau>a_{\ell+1}\}\subseteq\{\tau>a_\ell\}$ and
$\ind_{\{\tau>a_\ell\}}$ is $\mathcal F_{a_\ell}$-measurable,
\begin{align*}
 n\delta^2m_{\ell+1}
 &\le\E[\ind_{\{\tau>a_\ell\}}\norm{r_{a_\ell+B}}^2]\\
 &=\E\!\left[\ind_{\{\tau>a_\ell\}}
      \E[\norm{r_{a_\ell+B}}^2\mid\mathcal F_{a_\ell}]\right]\\
 &\le\Bigl(e^{-2J}+\frac\nu{1-\nu}\Bigr)n\delta^2m_\ell+
           \E[\ind_{\{\tau>a_\ell\}}E_{a_\ell}^{\rm tail}].
\end{align*}
Split the tail by \eqref{eq:tail-split}. For the deterministic part,
Lemma~\ref{lem:tail-det} and the strict discrepancy inequality on
$\{\tau>a_\ell\}$ imply
\[
 (1+\vartheta)\E[\ind_{\{\tau>a_\ell\}}\bar E_{a_\ell}^{\rm tail}]
 \le(1+\vartheta)\omega_bn\delta^2\Prob(\tau>a_\ell)
 \le\frac{(1+\vartheta)\omega_b}{\kstop^2}n\delta^2m_\ell.
\]
For the stochastic part remove the indicator, using nonnegativity,
and apply Lemma~\ref{lem:tail-subcritical}:
\[
 (1+\vartheta^{-1})\E[\ind_{\{\tau>a_\ell\}}\widetilde E_{a_\ell}^{\rm tail}]
 \le(1+\vartheta^{-1})C_G\gamma\mu^2n\delta^2s^{c_G}.
\]
Divide by $n\delta^2$ and use the definition of $q_\star$ in
\eqref{eq:horizon}.
\end{proof}

Iterating the recursion bounds the survival energies at all checkpoints,
in particular at the horizon. This is done in Lemma~\ref{lem:layer-sum},
where the bounds are used.

\subsection{Termination and a supermartingale}\label{sec:termination}

Up to the horizon $N$, the block contraction controls the stopping index.
Beyond it, a different argument takes over. As long as the rule has not
stopped, the residual exceeds the threshold, and the squared error then
decreases in conditional expectation by a fixed amount per step. This
drift gives termination, makes the stopped squared error a
supermartingale, and is the source of the condition $\kstop>1/(1-\nu)$ of
Theorem~\ref{thm:main}. It needs neither the source condition nor the
spectral decay.

\begin{proposition}[Drift of the error and termination]\label{prop:termination}
For deterministic data and initialization, let $\gamma R_x^2<2$ and
$\kstop>2/(2-\gamma R_x^2)$, and set
$c_\tau=(2-\gamma R_x^2)\kstop^2-2\kstop>0$. Then, for every $k\ge0$,
\begin{equation}\label{eq:drift}
 \E\bigl[\norm{\eta_{k+1}}^2-\norm{\eta_k}^2\bigm|\mathcal F_k\bigr]
 \le-\gamma\delta^2c_\tau\qquad\text{on }\{\tau>k\}.
\end{equation}
Consequently $Y_k=\norm{\eta_{k\wedge\tau}}^2$ is a nonnegative
supermartingale, and
\[
 \E\tau\le\frac{\norm{\eta_0}^2}{\gamma\delta^2c_\tau}<\infty;
\]
in particular $\tau<\infty$ almost surely. Under the hypotheses of
Theorem~\ref{thm:main}, these conditions hold and
$c_\tau\ge2\kstop g_{\mathrm d}>0$.
\end{proposition}
\begin{proof}
The update and $X\eta_k=r_k+\eps$ imply
\begin{align*}
 \E[\norm{\eta_{k+1}}^2-\norm{\eta_k}^2\mid\mathcal F_k]
 &=-\frac{2\gamma}{n}\langle r_k+\eps,r_k\rangle
   +\frac{\gamma^2}{n}\sum_i\norm{x_i}^2(r_k)_i^2\\
 &\le-\frac\gamma n
 [(2-\gamma R_x^2)\norm{r_k}^2-2\sqrt n\,\delta\norm{r_k}].
\end{align*}
On $\{\tau>k\}$, $\norm{r_k}>\kstop\sqrt n\,\delta$, and the bracket is
at least $n\delta^2c_\tau$, since $az^2-2z$ is increasing for
$z\ge\kstop>2/a$, with $a=2-\gamma R_x^2$. This proves \eqref{eq:drift}.
The pointwise identity
\[
 Y_{k+1}-Y_k
 =\ind_{\{\tau>k\}}\bigl(\norm{\eta_{k+1}}^2-\norm{\eta_k}^2\bigr),
\]
with an $\mathcal F_k$-measurable indicator, gives
$\E[Y_{k+1}-Y_k\mid\mathcal F_k]\le-\gamma\delta^2c_\tau\ind_{\{\tau>k\}}\le0$.
Taking expectations and summing over $0\le k<H$,
\[
 0\le\E Y_H
 \le\norm{\eta_0}^2-\gamma\delta^2c_\tau\sum_{k=0}^{H-1}\Prob(\tau>k)
 =\norm{\eta_0}^2-\gamma\delta^2c_\tau\E(\tau\wedge H).
\]
Let $H\to\infty$ and use monotone convergence on $\tau\wedge H$.
Under the hypotheses of Theorem~\ref{thm:main},
\eqref{eq:step-consequences} gives $\gamma R_x^2\le2\nu$, hence
$2/(2-\gamma R_x^2)\le1/(1-\nu)<\kstop$ and
$c_\tau\ge2(1-\nu)\kstop^2-2\kstop=2\kstop g_{\mathrm d}$.
\end{proof}

\subsection{Reconstruction at the stopping index}
\label{sec:thm2}
\label{sec:reconstruction}

The stopping rule controls the residual, whereas reconstruction
requires control of the parameter. A spectral split at a level
$T'^{-1}$ converts residual control into parameter control on the
eigenvalues at least $T'^{-1}$, at the cost of a factor $T'$ in squared
norm. On the smaller
eigenvalues, the source condition bounds the initial error and a
martingale maximal inequality controls the accumulated sampling noise
up to a deterministic horizon. We place the split at the end of the
block in which the rule stops, so that its cost is proportional to the
block index. By the survival recursion of
Subsection~\ref{sec:localization}, the probability of reaching a block
decays geometrically, and this cost is summable. On the late event
$\{\tau>N\}$, the error at $\tau$ is reduced to the error at the
deterministic index $N$ by the supermartingale of
Subsection~\ref{sec:termination}.

Let $L=T/\tprior$ and $\ell_s=\log(1/s)\ge1$. The rounding condition
$\tprior\ge\gamma$ gives
\begin{equation}\label{eq:L-bounds}
 C_b\tprior\le t_b\le(C_b+1)\tprior,\qquad
 1\le L\le J(C_b+1)\bar D\ell_s.
\end{equation}
For a level $T'>0$, split the parameter spectrum at $T'^{-1}$:
\[
 \mathcal H_{T'}=\{j\le m:\lambda_j\ge T'^{-1}\},\qquad
 \mathcal T_{T'}=\{j\le m:\lambda_j<T'^{-1}\}.
\]
These sets only contain the nonzero spectrum, since the parameter error
stays in $\mathcal R(X^\top)$. Write
$\eta_{j,k}=\langle\eta_k,v_j\rangle$ and
$\widetilde\eta_{j,k}=\eta_{j,k}-\E\eta_{j,k}$.
We use repeatedly the identities
$M^2=\rho^2\tprior^{-2r}=\delta^2\tprior$,
$\tprior^{-\alpha}=s^{2c_G}=e^{-2c_G\ell_s}$ and
$\tprior^{-d_\beta}=e^{-a_d\ell_s}$, where
\begin{equation}\label{eq:ad}
 a_d:=\frac{2d_\beta}{2r+1}=c_G+a_{\rm gap}>2c_G,
\end{equation}
with $a_{\rm gap}=(d_\beta+2r)/(2r+1)$ as in \eqref{eq:tail-constant};
the inequality holds because $r>0$. We also use the elementary bound
\begin{equation}\label{eq:log-absorption}
 \sup_{x\ge0}x^ke^{-ax}=\Bigl(\frac{k}{ea}\Bigr)^k,\qquad k,a>0.
\end{equation}

\begin{lemma}[Maximal fluctuations on the parameter tail]\label{lem:maximal}
Under Assumption~\ref{ass:main}, let $N\ge1$ be any deterministic integer
with $T=N\gamma\ge1$ (not necessarily the horizon of
\eqref{eq:horizon}), and let $\mathcal T_T$ be the tail set at the
level $T$. Then
\begin{equation}\label{eq:maximal-tail}
 \E\left[\sum_{j\in\mathcal T_T}
           \max_{0\le k\le N}\widetilde\eta_{j,k}^2\right]
 \le B_{\rm max}\gamma\mu^2(\rho^2+\delta^2T)T^{-d_\beta}.
\end{equation}
The constant $B_{\rm max}$ depends only on $r,\beta,C_\lambda$;
an explicit choice is given in \eqref{eq:maximal-constant}.
\end{lemma}
\begin{proof}
Multiplying a fluctuation coordinate by the inverse mean filter
removes its deterministic damping and produces a martingale.
Since $\widetilde\eta_{j,k}=\widetilde r_{j,k}/\sqrt{n\lambda_j}$, put
\[
 a_j=1-\gamma\lambda_j,\qquad
 Z_{j,k}=\sum_{a=0}^{k-1}a_j^{-(a+1)}
                \frac{\Delta_{j,a+1}}{\sqrt{n\lambda_j}}.
\]
Each multiplier $a_j^{-(a+1)}$ is deterministic.
By \eqref{eq:innovation}, $Z_{j,k}$ is adapted with zero conditional
mean increments.
Lemma~\ref{lem:duhamel} gives
$\widetilde\eta_{j,k}=a_j^k Z_{j,k}$. Here $a_j\ge1/2$, so the
inverse powers are well defined. For $j\in\mathcal T_T$ and $a+1\le N$,
using $-\log(1-x)\le x/(1-x)$,
\begin{equation}\label{eq:inverse-multiplier}
 a_j^{-2(a+1)}
 \le\exp\left(\frac{2N\gamma\lambda_j}{1-\gamma\lambda_j}\right)
 \le e^{4},
\end{equation}
because $N\gamma\lambda_j=T\lambda_j<1$ and
$\gamma\lambda_j\le1/2$.
Doob's $L^2$ maximal inequality, orthogonality of the increments, and
\eqref{eq:innovation-variance} now give
\begin{align*}
 \E\max_{k\le N}\widetilde\eta_{j,k}^2
 &\le\E\max_{k\le N}Z_{j,k}^2
 \le4\E Z_{j,N}^2\\
 &=\frac4{n\lambda_j}
       \sum_{a=0}^{N-1}a_j^{-2(a+1)}\E\Delta_{j,a+1}^2\\
 &\le\frac{4e^{4}\gamma^2\mu_j^2\lambda_j}{n}
             \sum_{a<N}\Lc_a.
\end{align*}
Sum over the finitely many $j\in\mathcal T_T$. By
\eqref{eq:cumulative-loss} and Lemma~\ref{lem:spectral-tail} with $p=1$,
\begin{align*}
 \E\sum_{j\in\mathcal T_T}\max_{k\le N}\widetilde\eta_{j,k}^2
 &\le\frac{4e^{4}\gamma\mu^2}{n}
       \left(\gamma\sum_{a<N}\Lc_a\right)
                   \sum_{j\in\mathcal T_T}\lambda_j\\
 &\le\frac{4e^{4}c_{1,\beta}}{1-\nu}
      \gamma\mu^2(C_\lambda^{2r}\rho^2+2\delta^2T)T^{-d_\beta}.
\end{align*}
Use $(1-\nu)^{-1}\le4/3$ to obtain $B_{\rm max}$.
Thus the fluctuations are controlled simultaneously at all
indices up to the deterministic horizon $N$.
\end{proof}

\begin{lemma}[Pathwise spectral split]\label{lem:pathwise}
Let $N'\ge1$ be a deterministic integer and $T'=N'\gamma$. For every
outcome and every integer $0\le k\le N'$,
\begin{equation}\label{eq:pathwise-split}
 \norm{\eta_k}^2\le\frac{2T'}{n}\bigl(\norm{r_k}^2+\norm\eps^2\bigr)
 +4\rho^2T'^{-2r}+4\delta^2\gamma k
 +2\sum_{j\in\mathcal T_{T'}}\max_{0\le i\le N'}\widetilde\eta_{j,i}^2 .
\end{equation}
Since the bound is pathwise, it also holds at any random index with
values in $\{0,\ldots,N'\}$.
\end{lemma}
\begin{proof}
Write $\bar\eta_{j,k}=\E\eta_{j,k}$. The relation
$r_{j,k}=\sqrt{n\lambda_j}\,\eta_{j,k}-\eps_j$ and
$\eta_{j,k}^2\le2\bar\eta_{j,k}^2+2\widetilde\eta_{j,k}^2$ give
\[
 \norm{\eta_k}^2\le\sum_{j\in\mathcal H_{T'}}
 \frac{(r_{j,k}+\eps_j)^2}{n\lambda_j}
 +2\sum_{j\in\mathcal T_{T'}}\bar\eta_{j,k}^2
 +2\sum_{j\in\mathcal T_{T'}}\widetilde\eta_{j,k}^2 .
\]
On $\mathcal H_{T'}$ one has $\lambda_j\ge T'^{-1}$, so the first sum is
at most $(T'/n)\norm{r_k+\eps}^2\le(2T'/n)(\norm{r_k}^2+\norm\eps^2)$.
By \eqref{eq:mean} and \eqref{eq:filter},
\begin{align*}
 2\sum_{j\in\mathcal T_{T'}}\bar\eta_{j,k}^2
 &\le4\sum_{j\in\mathcal T_{T'}}
       (1-\gamma\lambda_j)^{2k}\lambda_j^{2r}w_j^2
    +\frac4n\sum_{j\in\mathcal T_{T'}}
       \frac{(1-(1-\gamma\lambda_j)^k)^2}{\lambda_j}\eps_j^2\\
 &\le4\rho^2T'^{-2r}+4\delta^2\gamma k ,
\end{align*}
since $\lambda_j^{2r}<T'^{-2r}$ on $\mathcal T_{T'}$ and
$(1-(1-\gamma\lambda_j)^k)^2\le k\gamma\lambda_j$.
Finally $\widetilde\eta_{j,k}^2\le\max_{i\le N'}\widetilde\eta_{j,i}^2$
for $k\le N'$.
\end{proof}

\paragraph{Survival probabilities at the checkpoints.}
Recall the checkpoints $a_\ell=(\ell+1)B$ and the survival energies
$m_\ell$ of Lemma~\ref{lem:recursion}. Partition $\{\tau\le N\}$ into
the blocks
\[
 \Lambda_0=\{\tau\le B\},\qquad
 \Lambda_\ell=\{\ell B<\tau\le a_\ell\},\quad 1\le\ell\le N_c .
\]

\begin{lemma}[Survival sums]\label{lem:layer-sum}
For $0\le\ell\le N_c$,
\begin{equation}\label{eq:m-ell}
 m_\ell\le q_\star^{\ell}B_0'\gamma\mu^2
      +\frac{(1+\vartheta^{-1})C_G\gamma\mu^2}{1-\bar q_\star}\,s^{c_G},
 \qquad
 \Prob(\tau>a_\ell)\le\frac{m_\ell}{\kstop^2},
\end{equation}
and in particular
\begin{equation}\label{eq:late-constants}
 m_{N_c}\le B_1\gamma\mu^2s^{c_G},\qquad
 \Prob(\tau>N)\le\frac{B_1}{\kstop^2}\gamma\mu^2s^{c_G},\qquad
 B_1=B_0'+\frac{(1+\vartheta^{-1})C_G}{1-\bar q_\star}.
\end{equation}
Moreover,
\begin{equation}\label{eq:layer-sum}
 \sum_{\ell=0}^{N_c}(\ell+1)\Prob(\Lambda_\ell)
 \le1+\sum_{\ell=0}^{N_c-1}\Prob(\tau>a_\ell)
 \le1+K_1\gamma\mu^2,
\end{equation}
where an explicit $K_1$ is given in \eqref{eq:K1}.
\end{lemma}
\begin{proof}
Lemma~\ref{lem:recursion} gives $m_0\le B_0'\gamma\mu^2$ and
$m_{\ell+1}\le q_\star m_\ell+(1+\vartheta^{-1})C_G\gamma\mu^2s^{c_G}$ for
$0\le\ell<N_c$. Iterating, for $0\le\ell\le N_c$,
\[
 m_\ell\le q_\star^\ell m_0
   +(1+\vartheta^{-1})C_G\gamma\mu^2s^{c_G}\sum_{i=0}^{\ell-1}q_\star^i
 \le q_\star^\ell B_0'\gamma\mu^2
   +\frac{(1+\vartheta^{-1})C_G\gamma\mu^2}{1-q_\star}s^{c_G},
\]
and $q_\star\le\bar q_\star$ by \eqref{eq:uniform-contraction}. For
$\ell=N_c$ use $q_\star^{N_c}\le s^{c_G}$ from \eqref{eq:horizon-bounds}.
On $\{\tau>a_\ell\}$ one has $\norm{r_{a_\ell}}^2>\kstop^2n\delta^2$,
hence $\kstop^2n\delta^2\Prob(\tau>a_\ell)
\le\E[\ind_{\{\tau>a_\ell\}}\norm{r_{a_\ell}}^2]=n\delta^2m_\ell$.
Since $a_{N_c}=N$, this gives the bound for $\Prob(\tau>N)$.

For \eqref{eq:layer-sum}, the events $\Lambda_\ell$ are disjoint with
union $\{\tau\le N\}$, and $\bigcup_{\ell\ge i}\Lambda_\ell=\{iB<\tau\le N\}$
for $1\le i\le N_c$. Writing $\ell+1=\sum_{i=0}^{\ell}1$ and
exchanging the order of summation,
\[
 \sum_{\ell=0}^{N_c}(\ell+1)\ind_{\Lambda_\ell}
 =\ind_{\{\tau\le N\}}+\sum_{i=1}^{N_c}\ind_{\{iB<\tau\le N\}}
 \le1+\sum_{\ell=0}^{N_c-1}\ind_{\{\tau>a_\ell\}},
\]
since $iB=a_{i-1}$. Take expectations and insert \eqref{eq:m-ell}:
\[
 \sum_{\ell=0}^{N_c-1}\Prob(\tau>a_\ell)
 \le\frac{\gamma\mu^2}{\kstop^2(1-\bar q_\star)}\Bigl[
 B_0'+(1+\vartheta^{-1})C_G\,N_c\,s^{c_G}\Bigr].
\]
By \eqref{eq:horizon-bounds} and \eqref{eq:log-absorption},
$N_cs^{c_G}\le\bar D\,\ell_se^{-c_G\ell_s}\le\bar D/(ec_G)$.
\end{proof}

\paragraph{The error at the horizon on the late event.}
On $\{\tau>N\}$ the discrepancy inequality is not available at $N$,
but the block recursion controls the residual energy restricted to
this event, and the spectral split still converts it into parameter
error.

\begin{lemma}[Parameter error at the horizon on the late event]
\label{lem:late-energy}
\begin{equation}\label{eq:late-energy}
 \E\bigl[\norm{\eta_N}^2\ind_{\{\tau>N\}}\bigr]\le K_3\gamma\mu^2M^2,
\end{equation}
where an explicit $K_3$ is given in \eqref{eq:K3}.
\end{lemma}
\begin{proof}
Apply Lemma~\ref{lem:pathwise} with $N'=N$ and $k=N$, multiply by
$\ind_{\{\tau>N\}}$ and take expectations. By the definition of
$m_{N_c}$ (recall $a_{N_c}=N$) and $\norm\eps^2\le n\delta^2$, the
first term contributes at most
$2T\delta^2(m_{N_c}+\Prob(\tau>N))\le2(1+\kstop^{-2})T\delta^2m_{N_c}$.
Since $T\ge\tprior$ and $L\ge1$, the deterministic second and third
terms satisfy $4\rho^2T^{-2r}+4\delta^2T\le4M^2+4LM^2\le8LM^2$, and
$\Prob(\tau>N)\le m_{N_c}/\kstop^2$. With $T\delta^2=LM^2$, these
three terms contribute at most
\[
 \frac{2\kstop^2+10}{\kstop^2}\,L\,m_{N_c}M^2
 \le\frac{(2\kstop^2+10)J(C_b+1)\bar D}{e\,c_G\,\kstop^2}
     B_1\gamma\mu^2M^2,
\]
because $L\le J(C_b+1)\bar D\ell_s$ by \eqref{eq:L-bounds},
$m_{N_c}\le B_1\gamma\mu^2e^{-c_G\ell_s}$ by
Lemma~\ref{lem:layer-sum}, and \eqref{eq:log-absorption} with $k=1$.
For the last term drop the indicator and apply Lemma~\ref{lem:maximal}
at the horizon $N$, which is admissible since $T\ge\tprior\ge1$:
\[
 2\E\sum_{j\in\mathcal T_T}\max_{i\le N}\widetilde\eta_{j,i}^2
 \le2B_{\rm max}\gamma\mu^2(\rho^2+\delta^2T)T^{-d_\beta}
 =2B_{\rm max}\gamma\mu^2M^2
   \bigl(L^{-d_\beta}\tprior^{-\alpha}+L^{1-d_\beta}\tprior^{-d_\beta}\bigr).
\]
The first term in the bracket is at most one. For the second,
$1-d_\beta=1/\beta$, $L^{1/\beta}\le(J(C_b+1)\bar D\ell_s)^{1/\beta}$ and
$\tprior^{-d_\beta}=e^{-a_d\ell_s}$, so \eqref{eq:log-absorption} with
$k=1/\beta$ bounds it by $(J(C_b+1)\bar D)^{1/\beta}(e\beta a_d)^{-1/\beta}$.
Adding the parts proves \eqref{eq:late-energy}.
\end{proof}

\paragraph{Layered error bound below the horizon.}
On $\{\tau\le N\}$ the spectral split is placed at the end of the block
in which the rule stops. Its cost is then proportional to the block
index, which is summable against the survival probabilities of
Lemma~\ref{lem:layer-sum}.

\begin{lemma}[Layered error bound]\label{lem:layers}
\begin{equation}\label{eq:layers}
 \E\bigl[\norm{\eta_\tau}^2\ind_{\{\tau\le N\}}\bigr]
 \le(2\kstop^2+10)J(C_b+1)M^2
 +\bigl[(2\kstop^2+10)J(C_b+1)K_1+K_2\bigr]\gamma\mu^2M^2,
\end{equation}
where an explicit $K_2$ is given in \eqref{eq:K2}.
\end{lemma}
\begin{proof}
For $0\le\ell\le N_c$ put $T_{\ell+1}=a_\ell\gamma=(\ell+1)J t_b$.
By \eqref{eq:grid-rounding} and $C_b\ge1$,
\begin{equation}\label{eq:T-ell}
 (\ell+1)\tprior\le T_{\ell+1}\le(\ell+1)J(C_b+1)\tprior,
 \qquad T_{\ell+1}\ge\tprior\ge1 .
\end{equation}
On $\Lambda_\ell$ one has $\tau\le a_\ell$ and
$\norm{r_\tau}\le\kstop\sqrt n\,\delta$. Lemma~\ref{lem:pathwise} with
$N'=a_\ell$, evaluated at the random index $\tau$, gives on
$\Lambda_\ell$
\[
 \norm{\eta_\tau}^2\le2(\kstop^2+1)\delta^2T_{\ell+1}
 +4\rho^2T_{\ell+1}^{-2r}+4\delta^2T_{\ell+1}+F^{(\ell)},
 \qquad
 F^{(\ell)}:=2\sum_{j\in\mathcal T_{T_{\ell+1}}}
          \max_{0\le i\le a_\ell}\widetilde\eta_{j,i}^2 ,
\]
where $\gamma\tau\le T_{\ell+1}$ was used. Only
$\norm{r_\tau}\le\kstop\sqrt n\,\delta$ enters, so jumping below the
threshold in discrete time causes no additional term.
By \eqref{eq:T-ell}, $\rho^2T_{\ell+1}^{-2r}\le M^2$ and
$\delta^2T_{\ell+1}\le(\ell+1)J(C_b+1)M^2$, so on $\Lambda_\ell$
\begin{equation}\label{eq:layer-pathwise}
 \norm{\eta_\tau}^2\le(2\kstop^2+10)(\ell+1)J(C_b+1)M^2+F^{(\ell)} .
\end{equation}
Multiply by $\ind_{\Lambda_\ell}$, take expectations and sum over
$\ell$. By \eqref{eq:layer-sum}, the first term contributes at most
$(2\kstop^2+10)J(C_b+1)(1+K_1\gamma\mu^2)M^2$.
For the second term drop the indicators and apply
Lemma~\ref{lem:maximal} at each horizon $a_\ell$; its tail set is then
$\mathcal T_{T_{\ell+1}}$:
\[
 \sum_{\ell=0}^{N_c}\E F^{(\ell)}
 \le2B_{\rm max}\gamma\mu^2\sum_{\ell=0}^{N_c}
   \bigl(\rho^2+\delta^2T_{\ell+1}\bigr)T_{\ell+1}^{-d_\beta}.
\]
By \eqref{eq:T-ell},
\[
 \rho^2T_{\ell+1}^{-d_\beta}\le M^2(\ell+1)^{-d_\beta}\tprior^{-\alpha},
 \qquad
 \delta^2T_{\ell+1}^{1-d_\beta}
 \le M^2(J(C_b+1))^{1/\beta}(\ell+1)^{1/\beta}\tprior^{-d_\beta}.
\]
Comparison with integrals gives, since $0<d_\beta<1$,
\[
 \sum_{\ell=0}^{N_c}(\ell+1)^{-d_\beta}
 \le\frac{(N_c+1)^{1-d_\beta}}{1-d_\beta}=\beta(N_c+1)^{1/\beta},
 \qquad
 \sum_{\ell=0}^{N_c}(\ell+1)^{1/\beta}\le(N_c+1)^{1+1/\beta}.
\]
With $N_c+1\le\bar D\ell_s$ from \eqref{eq:horizon-bounds},
$\tprior^{-\alpha}=e^{-2c_G\ell_s}$, $\tprior^{-d_\beta}=e^{-a_d\ell_s}$
and \eqref{eq:log-absorption} (with $k=1/\beta$, $a=2c_G$, and with
$k=1+1/\beta$, $a=a_d$),
\[
 \beta(N_c+1)^{1/\beta}\tprior^{-\alpha}
 \le\beta\bar D^{1/\beta}(2e\beta c_G)^{-1/\beta},
 \qquad
 (N_c+1)^{1+1/\beta}\tprior^{-d_\beta}
 \le\bar D^{1+1/\beta}\Bigl(\frac{1+1/\beta}{e\,a_d}\Bigr)^{1+1/\beta}.
\]
Hence $\sum_\ell\E F^{(\ell)}\le K_2\gamma\mu^2M^2$.
\end{proof}

\paragraph{Proof of the main theorem.}

\begin{proof}[Proof of Theorem~\ref{thm:main}]
\emph{Termination and a supermartingale.}
By Proposition~\ref{prop:termination}, $\tau<\infty$ almost surely; on
the null set $\{\tau=\infty\}$ set $\eta_\tau:=\eta_0$, in accordance with
the convention $\theta_\tau=\theta_0$ there. The stopped squared
error $Y_k=\norm{\eta_{k\wedge\tau}}^2$ is a nonnegative supermartingale
with
\begin{equation}\label{eq:restart-drift}
 \E[Y_{k+1}-Y_k\mid\mathcal F_k]
 \le-\gamma\delta^2\bar c_\tau\ind_{\{\tau>k\}},\qquad
 \bar c_\tau:=2\kstop g_{\mathrm d}>0.
\end{equation}

\emph{Mean-square error.} Split
\[
 \E\norm{\eta_\tau}^2
 =\E\bigl[\norm{\eta_\tau}^2\ind_{\{\tau\le N\}}\bigr]
 +\E\bigl[\norm{\eta_\tau}^2\ind_{\{\tau>N\}}\bigr].
\]
Lemma~\ref{lem:layers} bounds the first term. For the second, let
$k\ge N$. Since $\{\tau>N\}\in\mathcal F_N$,
\[
 \E\bigl[Y_k\ind_{\{\tau>N\}}\bigr]
 =\E\bigl[\ind_{\{\tau>N\}}\E[Y_k\mid\mathcal F_N]\bigr]
 \le\E\bigl[\norm{\eta_N}^2\ind_{\{\tau>N\}}\bigr].
\]
Since $\tau<\infty$ almost surely, $Y_k=\norm{\eta_\tau}^2$ for all
$k\ge\tau$, so $Y_k\ind_{\{\tau>N\}}\to\norm{\eta_\tau}^2\ind_{\{\tau>N\}}$
almost surely. Fatou's lemma and Lemma~\ref{lem:late-energy} give
\[
 \E\bigl[\norm{\eta_\tau}^2\ind_{\{\tau>N\}}\bigr]
 \le\liminf_{k\to\infty}\E\bigl[Y_k\ind_{\{\tau>N\}}\bigr]
 \le K_3\gamma\mu^2M^2 .
\]
Adding the two bounds and using
$M^2=\delta^2\tprior=\rho^{2/(2r+1)}\delta^{4r/(2r+1)}$ proves the first
bound in \eqref{eq:main-expectation} with $C_E=(2\kstop^2+10)J(C_b+1)$
and $\widetilde C_E=(2\kstop^2+10)J(C_b+1)K_1+K_2+K_3$.

\emph{Expected stopping index.} Write $\tau=\tau\wedge N+(\tau-N)^+$.
Pointwise, $\tau\wedge N=\sum_{k=0}^{N-1}\ind_{\{\tau>k\}}$. Grouping the
indices $k\in[\ell B,(\ell+1)B)$ for $0\le\ell\le N_c$ and using
$\ind_{\{\tau>k\}}\le\ind_{\{\tau>\ell B\}}$ there,
\[
 \tau\wedge N\le B\sum_{\ell=0}^{N_c}\ind_{\{\tau>\ell B\}}
 \le B\Bigl(1+\sum_{\ell=0}^{N_c-1}\ind_{\{\tau>a_\ell\}}\Bigr).
\]
By \eqref{eq:layer-sum}, $\E[\tau\wedge N]\le B(1+K_1\gamma\mu^2)$, and
$\gamma B=J t_b\le J(C_b+1)\tprior$ by \eqref{eq:grid-rounding}.
Beyond the horizon, let $H>N$ be an integer. For $N\le k<H$ the event
$\{\tau>N\}$ is $\mathcal F_k$-measurable and
$\{\tau>k\}\subseteq\{\tau>N\}$, so \eqref{eq:restart-drift} gives
$\E[(Y_{k+1}-Y_k)\ind_{\{\tau>N\}}]\le-\gamma\delta^2\bar c_\tau\Prob(\tau>k)$.
Summing over $N\le k<H$ and using $Y_H\ge0$,
\[
 \gamma\delta^2\bar c_\tau\sum_{k=N}^{H-1}\Prob(\tau>k)
 \le\E\bigl[\norm{\eta_N}^2\ind_{\{\tau>N\}}\bigr].
\]
Pointwise, $\sum_{k=N}^{H-1}\ind_{\{\tau>k\}}=(\tau\wedge H-N)^+$, which
increases to $(\tau-N)^+$ as $H\to\infty$. Monotone convergence,
Lemma~\ref{lem:late-energy} and $M^2=\delta^2\tprior$ yield
\[
 \gamma\,\E\bigl[(\tau-N)^+\bigr]
 \le\frac{K_3\gamma\mu^2M^2}{\delta^2\bar c_\tau}
 =\frac{K_3}{\bar c_\tau}\gamma\mu^2\tprior .
\]
Adding the two parts proves the second bound in
\eqref{eq:main-expectation} with $C_\tau=J(C_b+1)$ and
$\widetilde C_\tau=J(C_b+1)K_1+K_3/\bar c_\tau$.
\end{proof}

\begin{remark}[Where the logarithm is avoided]\label{rem:expectation}
The horizon $T=L\tprior$, with $L$ of order $\log(\rho/\delta)$,
enters only in products with quantities of order $(\delta/\rho)^{c_G}$:
the surviving residual energy $m_{N_c}$ and the tail fluctuations.
These products are bounded by \eqref{eq:log-absorption}. Below the
horizon, the spectral split adapts to the block in which the rule
stops, and the cost $(\ell+1)M^2$ of block $\ell$ is summable against
the survival probabilities. A single split at the level $T^{-1}$ for
every stopping index below $N$, as in \cite{jahn2026sde}, would instead
cost the factor $L$, and hence a factor $\log(\rho/\delta)$ in the squared
error.
\end{remark}

\section{Numerical experiments}\label{sec:numerics}

The incoherence parameters \eqref{eq:incoherence} were introduced in the
earlier version of this article \cite{jahn2026sde}. Section~3 there
evaluates them on a
discretized integral operator and on two data-driven designs, and applies
the discrepancy principle to the iteration \eqref{eq:sgd} on the Phillips
test problem. At $\nu=1/8$ and with $\kstop=1.2$, stochastic gradient
descent attained the accuracy of Landweber iteration stopped by the same
rule with $42$ to $45$ times fewer row accesses. This held for the ground
truth shipped with that discretization, but not for a smoothed one that
violates \ref{ass:compatibility} substantially. Although the analysis
there concerns a continuous-time model, the parameters of that experiment
satisfy the conditions $\nu\le1/4$ and $\kstop>1/(1-\nu)=8/7$ of
Theorem~\ref{thm:main} for the iteration itself. We do not repeat these
experiments. Instead we examine
a question that is specific to the discrete iteration and left open by
Remark~\ref{rem:threshold}: how large the safety factor has to be for a
given step size.

\begin{figure}[t]
\centering
\includegraphics[width=\textwidth]{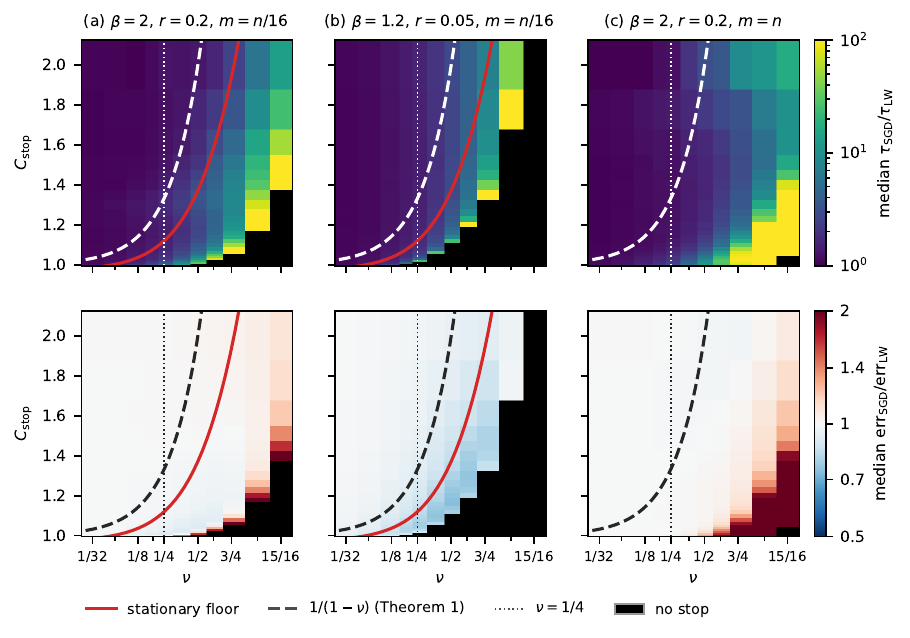}
\caption{Stochastic gradient descent against Landweber iteration with the
same step and the same rule, over $\nu=\gamma\Tr K/2$ (logit scale) and
$\kstop$, with $256$ runs per cell. Top: median ratio
$\tau_{\mathrm{SGD}}/\tau_{\mathrm{LW}}$ of the stopping indices. Bottom:
median ratio of the errors at stopping. Hadamard designs with
$\lambda_j=j^{-\beta}$: (a), (b) $n=4096$, $m=256$; (c) $n=m=1024$.
Black: more than half of the runs did not stop within
$100\,\tprior/\gamma$ iterations. Red: the stationary floor
$\norm{P\eps}/(\sqrt n\,\delta\sqrt{1-\nu})$, with $P$ the orthogonal
projection onto $\mathcal N(X^\top)$, absent in (c). Dashed: the
threshold $1/(1-\nu)$ of Theorem~\ref{thm:main}. Dotted: $\nu=1/4$, the
limit of \ref{ass:step}.}
\label{fig:interplay}
\end{figure}

\paragraph{Question.}
For Landweber iteration with deterministic noise, every safety factor
$\kstop>1$ is admissible. Theorem~\ref{thm:main} requires
$\kstop>1/(1-\nu)$, and Remark~\ref{rem:threshold} explains why some
excess over one is needed. For rows of equal norm, if the second moments
of the unstopped residual converge, the identity
\eqref{eq:exact-second-moment} gives the limit
$\E\norm{r}^2=\norm{P\eps}^2/(1-\gamma\kappa/2)$, with $P$ the orthogonal
projection onto $\mathcal N(X^\top)$: the limit $\norm{P\eps}^2$ of the
Landweber residual, inflated by the sampling noise. If most of the noise lies in $\mathcal N(X^\top)$, a safety factor
just above one lies below this stationary floor, and the rule can stop
only through downward fluctuations of the residual. We test whether such
safety factors are nevertheless admissible.

\paragraph{Setup.}
The eigenvectors $u_1,\dots,u_m$ of $K$ are $m=256$ randomly chosen
columns of the normalized Sylvester--Hadamard matrix of order $n=4096$.
Hence $\mu_j^2=1$ for all $j$, every row has squared norm $\kappa=\Tr K$,
and $\nu=\gamma\kappa/2$. In particular, $\nu<1$ is the stability
condition $\gamma\norm{x_i}^2<2$ of the relaxed Kaczmarz method, and
$\nu=1/2$ is the classical Kaczmarz method. Since $m/n=1/16$, about
$94\%$ of the noise lies in $\mathcal N(X^\top)$. This is a stress test:
the stationary floor then dominates the residual near the oracle time. We
take $\lambda_j=j^{-\beta}$ and $\eta_0=\Sigma^rv$ with $v$ uniform on the
sphere of radius $\rho=1$, in two configurations satisfying
\ref{ass:compatibility}: $\beta=2$, $r=0.2$, $\delta=10^{-2}$, and
$\beta=1.2$, $r=0.05$, $\delta=10^{-3/2}$; the smaller $r$ for $\beta=1.2$
is forced by \ref{ass:compatibility}, which requires $r<1/12$ there. A
third configuration, a square design with $n=m=1024$, $\beta=2$, $r=0.2$
and $\delta=10^{-2}$, has no null space. The noise is Gaussian, rescaled to
$\norm\eps=\sqrt n\,\delta$. For each
$\nu\in\{1/32,1/16,1/8,3/16,1/4,3/8,1/2,5/8,3/4,7/8,15/16\}$, which
extends beyond \ref{ass:step} up to the stability limit, we run $256$
independent copies of \eqref{eq:sgd} with $\gamma=2\nu/\kappa$, each for
$100\,\tprior/\gamma$ iterations. For every $\kstop$ on a grid in $[1,2]$
we record the first index at which $\norm{r_k}\le\kstop\sqrt n\,\delta$.
Runs that do not stop within this budget count as not stopped, and their
error is taken at the end of the budget.
All quantities are diagonal in the eigenbasis, so the iteration is
simulated exactly in these coordinates. The reference is Landweber
iteration with the same step and the same rule. It is the mean of the
stochastic iteration, and it stops by time
$\gamma\tau_{\mathrm{LW}}\le2.1\,\tprior$ for every $\kstop\ge1$.

\paragraph{Results.}
In the two designs with a null space, the average of $\norm{r_k}^2$ over
the second half of the iterations and over the runs agrees with
$\norm{P\eps}^2/(1-\nu)$ to within $0.3\%$ for every $\nu$, which confirms
the stationary identity. The bottom row of Figure~\ref{fig:interplay}
compares the errors with those of Landweber iteration stopped by the same
rule, and thus isolates the effect of the sampling noise on the rule. The
rule itself is coarse in these designs, since most of the threshold is
taken up by noise that cannot be fitted: in the range of
Theorem~\ref{thm:main}, Landweber iteration stopped by the rule has $1.4$
to $1.7$ times its best error for $\beta=2$, and $1.8$ to $5$ times for
$\beta=1.2$.

In the range of Theorem~\ref{thm:main}, $\nu\le1/4$ and
$\kstop>1/(1-\nu)$, all runs stopped in all three configurations. The
median stopping index was less than twice that of Landweber iteration, and
the median error was within $1\%$ of Landweber's or smaller. The same
holds well beyond this range. Above the stationary floor, every run
stopped, the median stopping index was at most $2.9$ ($\beta=2$) and $7.4$
($\beta=1.2$) times that of Landweber iteration, and the errors were
within $3\%$ of Landweber's or smaller. The behavior changes at the floor,
about $(1-\nu)^{-1/2}$ here, rather than at the threshold $1/(1-\nu)$.

Below the floor the outcome depends on the spectrum. For $\beta=2$ the
first eigenvalue carries $61\%$ of $\Tr K$, so the fluctuations of the
residual are concentrated in few directions and deep downward excursions
are frequent. The rule then stops for every $\kstop\ge1.01$ up to
$\nu=1/2$, but late: at $\nu=1/2$ and $\kstop=1.01$ after $31$ times as
many iterations as Landweber iteration and with $1.7$ times its error. For
$\beta=1.2$ the fluctuations are spread over many directions, and
excursions deep enough to reach the threshold are rare. A factor of $1.01$
fails already at $\nu=1/4$, where $98\%$ of the runs did not stop, and
$\kstop=1.1$ fails at $\nu=1/2$. Where the rule does stop for $\beta=1.2$,
its error is at most Landweber's, because the later stop compensates for
the coarse rule. Where it does not stop, the error at the end of the
budget reaches $4.6$ ($\beta=2$) and $5.2$ ($\beta=1.2$) times
Landweber's.

In the square design the residual eventually falls below every threshold,
because the noise is fitted. Within the budget all runs stopped, except at
$\nu=15/16$, where more than half did not stop for $\kstop\le1.04$. For
$\nu\le1/4$ the stopping index was at most $2.7$ times that of Landweber
iteration for every $\kstop\ge1$, with errors within $3\%$ of Landweber's.
For larger $\nu$ and $\kstop$ close to one, however, the noise in the slow
directions plays the role of the null space, and the rule stops only after
part of it has been fitted: at $\nu=3/4$ and $\kstop=1.05$ after $90$
times as many iterations and with twice the error, and at $\nu=15/16$ with
more than ten times the error.

In short, the admissible safety factors depend on the step size. A factor
merely above one, which suffices for Landweber iteration, does not suffice
for constant-step stochastic gradient descent. The condition
$\kstop>1/(1-\nu)$ of Theorem~\ref{thm:main} lies well inside the region
where the rule works; with a null space, the rule works at least down to
the stationary floor. Below the floor the outcome depends on the design,
so a condition that is uniform over designs should not be expected there.

\section*{Funding}
T. Jahn is funded by the Deutsche Forschungsgemeinschaft (DFG, German
Research Foundation) under Germany's Excellence Strategy -- The Berlin
Mathematics Research Center MATH+ (EXC-2046/1, EXC-2046/2, project ID:
390685689). A. Schertzer has received funding from the European Union's
Horizon Europe research and innovation programme under the Marie
Sk\l{}odowska-Curie grant agreement No.~101203974 (SLOHD).

\appendix
\section{Remaining proofs}\label{sec:proofs}

We first record the spectral estimates used in the fluctuation and
tail bounds. The subsequent proofs follow the order of the analysis, and
Appendix~\ref{sec:constants} collects the explicit constants.

\begin{lemma}[Spectral sums]\label{lem:spectral-tail}
Under \ref{ass:spectrum}, for $p\in\{1,2\}$ define
\begin{equation}\label{eq:spectral-constants}
 c_{p,\beta}:=\max\left\{
 C_\lambda^p\zeta(p\beta),\
 1+C_\lambda^{1/\beta}+\frac{C_\lambda^{1/\beta}}{p\beta-1}
 \right\}.
\end{equation}
Then, for every $t>0$,
\begin{align}
 \sum_{j:\lambda_j<t^{-1}}\lambda_j^p
 &\le c_{p,\beta}(1\vee t)^{-(p-1/\beta)},\label{eq:spectral-tail}\\
 \sum_{j\le m}\lambda_j^p e^{-2t\lambda_j}
 &\le c_{p,\beta}(1\vee t)^{-(p-1/\beta)}.\label{eq:spectral-exponential}
\end{align}
The exponential estimate also holds at $t=0$.
\end{lemma}
\begin{proof}
For $t\le1$, both sums are bounded by
$\sum_j C_\lambda^p j^{-p\beta}=C_\lambda^p\zeta(p\beta)$;
the series is finite since $p\beta>1$.
For $t\ge1$, put $D=\lceil(C_\lambda t)^{1/\beta}\rceil$
and extend the spectrum by zeros if necessary.
The indices $j>D$ contribute at most
\[
 C_\lambda^p\sum_{j>D}j^{-p\beta}
 \le\frac{C_\lambda^p}{p\beta-1}D^{1-p\beta}
 \le\frac{C_\lambda^{1/\beta}}{p\beta-1}
                           t^{-(p-1/\beta)}.
\]
Among $j\le D$, each term in the tail sum is at most $t^{-p}$.
For the exponential sum use
$\lambda^p e^{-2t\lambda}\le(p/(2e))^p t^{-p}\le t^{-p}$
for $p=1,2$.
In either case the contribution is at most
\[
 Dt^{-p}\le(1+(C_\lambda t)^{1/\beta})t^{-p}
 \le(1+C_\lambda^{1/\beta})t^{-(p-1/\beta)}.
\]
Adding the two parts proves the assertions.
\end{proof}

\subsection{Fluctuation estimates}\label{sec:proofs-fluctuations}

\begin{proof}[Proof of Lemma~\ref{lem:integrated-loss}]
Summing \eqref{eq:volterra} for $0\le k<H$ gives, for $H\ge1$,
\begin{align*}
 \sum_{k=0}^{H-1}\Lc_k
 &\le\sum_{k=0}^{H-1}\varphi_k+
       \sum_{a=0}^{H-2}\Lc_a\sum_{\ell=1}^{H-1-a}h_\ell\\
 &\le\sum_{k=0}^{H-1}\varphi_k+\nu\sum_{a=0}^{H-1}\Lc_a.
\end{align*}
Rearrangement gives
$(1-\nu)\sum_{k<H}\Lc_k\le\sum_{k<H}\varphi_k$.
For $H=0$ the assertion is immediate.
Writing $w_j=\langle v,v_j\rangle$ and $q_j=1-2\gamma\lambda_j$, we have
\[
 \varphi_k\le2n\sum_{j\le m}q_j^k\lambda_j^{2r+1}w_j^2
                  +2\sum_{j\le n}q_j^k\eps_j^2.
\]
Consequently,
\[
 \gamma\sum_{k<H}\varphi_k
 \le2n\gamma\sum_{j\le m}
        \frac{\lambda_j^{2r+1}w_j^2}{2\gamma\lambda_j}
        +2n\delta^2H\gamma
 \le n C_\lambda^{2r}\rho^2+2n\delta^2H\gamma.
\]
This proves \eqref{eq:cumulative-loss}.

\end{proof}

\begin{proof}[Proof of Lemma~\ref{lem:resolvent-decay}]
For $\ell\ge1$ and a nonzero eigenvalue,
\[
 (1-\gamma\lambda)^{2(\ell-1)}
 \le e^{-2\gamma\lambda(\ell-1)}
 \le e\,e^{-2\gamma\ell\lambda},
\]
since $2\gamma\lambda\le1$.
Apply \eqref{eq:spectral-exponential} with $p=2$ to prove the kernel
bound. This also handles the lag $\ell=1$ without losing a factor
depending on $\gamma$.

In the sum defining $(\widehat h^{*j})_k$, the positive integer lags
$\ell_1,\ldots,\ell_j$ sum to $k$. At least one is at least $k/j$.
For each possible choice of that coordinate, bound its factor by the
supremum over lags at least $k/j$, then sum all the other coordinates
independently over positive integers. Nonnegativity gives
\[
 (\widehat h^{*j})_k
 \le j\,\widehat\nu^{j-1}
             \sup_{\ell\ge k/j,\ \ell\in\mathbb N}\widehat h_\ell
 \le e\,c_{2,\beta}\gamma^2\mu^2
       j^{p_\beta+1}\widehat\nu^{j-1}
                       (1\vee k\gamma)^{-p_\beta}.
\]
We used $1\vee(k\gamma/j)\ge(1\vee k\gamma)/j$.
For $j>k$ the left side is zero. Summing and using
$\widehat\nu\le1/3$ proves the result.
\end{proof}

\begin{proof}[Proof of Proposition~\ref{prop:variance}]
At $k=0$, $V_0=0$. For $k\ge1$ put $t=k\gamma$ and $u=\lfloor k/2\rfloor$.
Split $(\widehat{\mathcal Q}*f)_k$ into $a\le u$ and $a>u$.
For the first part, $k-a\ge k/2$, and
$u+1\le k$. Lemma~\ref{lem:resolvent-decay} and
\eqref{eq:mean-cumulative} imply
\[
 \sum_{a=0}^u\widehat{\mathcal Q}_{k-a}f_a
 \le D_{\mathcal Q}\gamma\mu^2 n(1\vee t/2)^{-p_\beta}
          \left[(4/3)C_\lambda^{2r}\rho^2+2\delta^2t\right].
\]
For the second part use monotonicity of $f$:
\[
 \sum_{a=u+1}^{k-1}\widehat{\mathcal Q}_{k-a}f_a
 \le\frac{\widehat\nu}{1-\widehat\nu}f_{u+1}.
\]
This is also valid for the empty sum when $k=1$.
Furthermore
\[
 \frac{\widehat\nu}{1-\widehat\nu}
 \le\frac{(2/3)\gamma\mu^2\kappa}{1-1/3}
 =\gamma\mu^2\kappa.
\]
Since $(u+1)\gamma\ge t/2$, \eqref{eq:mean-envelope} bounds this part by
$2\gamma\mu^2\bar\kappa n
[B_r\rho^2(1\vee t/2)^{-(2r+1)}+\delta^2]$.
Use $(1\vee t/2)^{-a}\le2^a(1\vee t)^{-a}$.
The compatibility condition gives $p_\beta>2r+1$, and
$t(1\vee t)^{-p_\beta}\le1$ because $p_\beta>1$.
The two bounds give the value of $C_V$ in \eqref{eq:variance-constant}.
\end{proof}

\subsection{Stopping-index estimates}\label{sec:proofs-stability}

\begin{proof}[Proof of Lemma~\ref{lem:horizon}]
Since $\nu>0$ and $\kstop<\infty$, clearly $g_{\mathrm c}<1$, and $g_{\mathrm d}>0$ is the
hypothesis $\kstop(1-\nu)>1$. For $g_{\mathrm c}>0$, note
$1-2\nu-(1-\nu)^3=\nu(1-3\nu+\nu^2)\ge0$ for $\nu\le1/4$, hence
\[
 \frac1{\kstop^2}<(1-\nu)^2\le\frac{1-2\nu}{1-\nu}
 =1-\frac\nu{1-\nu}.
\]
Since $g_{\mathrm c}<1$, $J\ge\lceil\frac12\log4\rceil=1$. The first term in the
definition of $C_b$ is at least one, so $C_b\ge1$.
For \eqref{eq:uniform-contraction}, the choice of $C_b$ gives
$C_b^{-(r+1/2)}\le\vartheta/3$, so
$\omega_b\le(1+\vartheta/3)^2\le1+\vartheta$ because $\vartheta\le1$. Hence
$(1+\vartheta)\omega_b\le(1+\vartheta)^2\le1+3\vartheta\le1+g_{\mathrm c}\kstop^2/4$.
Together with $e^{-2J}\le g_{\mathrm c}/4$,
\[
 q_\star\le\frac{g_{\mathrm c}}4+\frac\nu{1-\nu}+\frac1{\kstop^2}+\frac{g_{\mathrm c}}4
 =1-\frac{g_{\mathrm c}}2=\bar q_\star .
\]
The ceiling inequalities and $\gamma\le\tprior$ give the bounds on
$t_b$. Since $\log q_\star<0$,
the ceiling in the definition of $N_c$ implies
\[
 N_c\log q_\star\le-c_G\log(1/s),\qquad
 N_c+1<\frac{c_G\log(1/s)}{|\log q_\star|}+2.
\]
Use $|\log q_\star|\ge|\log\bar q_\star|$ and $\log(1/s)\ge1$.
Finally $T=(N_c+1)J t_b$, and $T\ge t_b\ge\tprior$; multiply the
bound for $N_c+1$ by the bound for $J t_b/\tprior$.
\end{proof}

\begin{proof}[Proof of Lemma~\ref{lem:block}]
On $\mathcal H_b$,
$(1-2\gamma\lambda_j)^B\le e^{-2\gamma J b\lambda_j}
=e^{-2J t_b\lambda_j}\le e^{-2J}$;
on $\mathcal T_b$ the multiplier is at most one. Hence
$\norm{Q^{B/2}r_a}^2\le e^{-2J}\norm{r_a}^2+E_a^{\rm tail}$.
Apply \eqref{eq:conditional-loss} with lag $B$; its remainder has
coefficient $\nu/(1-\nu)$.
\end{proof}

\begin{proof}[Proof of Lemma~\ref{lem:tail-det}]
By the mean formula, for $j\in\mathcal T_b$,
$\bar r_{j,a}=(1-\gamma\lambda_j)^a(\sqrt n\,\lambda_j^{r+1/2}w_j-\eps_j)$
with $0\le1-\gamma\lambda_j\le1$. The triangle inequality in
$\ell^2(\mathcal T_b)$ gives
\[
 (\bar E_a^{\rm tail})^{1/2}
 \le\Bigl(n\sum_{j\in\mathcal T_b}\lambda_j^{2r+1}w_j^2\Bigr)^{1/2}
   +\norm\eps
 \le\sqrt n\,\rho\,t_b^{-(r+1/2)}+\sqrt n\,\delta
 \le\sqrt n\,\delta\bigl(C_b^{-(r+1/2)}+1\bigr),
\]
since $t_b\ge C_b\tprior$ and $\rho\tprior^{-(r+1/2)}=\delta$.
This includes the frozen zero modes: they contribute only to the
noise term, since $w_j=0$ for $j>m$.
\end{proof}

\begin{proof}[Proof of Lemma~\ref{lem:tail-subcritical}]
At a deterministic $a$, Lemma~\ref{lem:duhamel} and the bounds on
the multipliers give
\[
 \E\widetilde E_a^{\rm tail}
 \le\gamma^2\mu^2
       \left(\sum_{j\in\mathcal T_b}\lambda_j^2\right)
       \sum_{z=0}^{a-1}\Lc_z.
\]
The cumulative loss bound and the spectral tail estimate with $p=2$
therefore imply, in fact for every $a\ge0$,
\[
 \E\widetilde E_a^{\rm tail}
 \le B_{\rm tail}\gamma\mu^2n\,t_b^{-(1+d_\beta)}
                                  (\rho^2+\delta^2a\gamma).
\]
For $a\le N=(N_c+1)J b$, one has $a\gamma\le(N_c+1)J t_b$,
and with $t_b\ge\tprior$ this is at most
\[
 B_{\rm tail}\gamma\mu^2n\delta^2
       \bigl[\tprior^{-\alpha}+(N_c+1)J\,\tprior^{-d_\beta}\bigr].
\]
Write $\ell_s=\log(1/s)\ge1$. The first term is
$\tprior^{-\alpha}=s^{2c_G}\le s^{c_G}$. For the second,
Lemma~\ref{lem:horizon} gives
\[
 (N_c+1)\tprior^{-d_\beta}
 \le\bar D\ell_s s^{2d_\beta/(2r+1)}
 =\bar D s^{c_G}\ell_s e^{-a_{\rm gap}\ell_s}
 \le\frac{\bar D}{e a_{\rm gap}}s^{c_G}.
\]
We used $\sup_{x\ge0}xe^{-a x}=1/(ea)$.
This proves the estimate with the constant in \eqref{eq:tail-constant}.
\end{proof}

\subsection{Explicit constants and quantifiers}\label{sec:constants}

We collect explicit choices for the constants used in the analysis.
This separates their numerical values from the estimates they support.
All expressions are finite for $r>0$, $\beta>1$, $2r<d_\beta$,
$\nu\le1/4$ and $\kstop>1/(1-\nu)$. The constants $c_{p,\beta}$ are
defined in \eqref{eq:spectral-constants}; their dependence on
$C_\lambda$ is implicit in the notation. The design and the step enter
only through $g_{\mathrm c}$ of \eqref{eq:contraction-constants} and
$g_{\mathrm d}$ of Theorem~\ref{thm:main}, via the
parameters $\vartheta,J,C_b$ of \eqref{eq:parameters}.

\paragraph{Resolvent and variance.}
Recall $p_\beta=1+d_\beta$, $\bar\kappa=C_\lambda\zeta(\beta)$,
and $B_r=\max\{C_\lambda^{2r+1},c_r^2\}$. Take
\begin{equation}\label{eq:resolvent-constant}
 D_{\mathcal Q}
 =e\,c_{2,\beta}
       \sum_{\ell\ge1}\ell^{p_\beta+1}3^{-(\ell-1)},
\end{equation}
and
\begin{equation}\label{eq:variance-constant}
 C_V=2\bar\kappa\max\{2^{2r+1}B_r,1\}
       +2^{p_\beta}D_{\mathcal Q}
                         \max\{(4/3)C_\lambda^{2r},2\}.
\end{equation}
These depend only on $r,\beta,C_\lambda$.

\paragraph{Stability and maximal fluctuations.}
Recall
\[
 \bar q_\star=1-\frac{g_{\mathrm c}}2,\qquad
 \bar D=\frac{c_G}{|\log\bar q_\star|}+2.
\]
The spectral-tail constants are
\begin{equation}\label{eq:tail-constant}
 \begin{gathered}
 B_{\rm tail}=\frac43c_{2,\beta}\max\{C_\lambda^{2r},2\},
 \qquad a_{\rm gap}=\frac{d_\beta+2r}{2r+1}>0,\\
 C_G=B_{\rm tail}
            \left(1+\frac{J\bar D}{e a_{\rm gap}}\right).
 \end{gathered}
\end{equation}
For the parameter-tail maximum, take
\begin{equation}\label{eq:maximal-constant}
 B_{\rm max}=\frac{16}{3}e^{4}c_{1,\beta}
                         \max\{C_\lambda^{2r},2\}.
\end{equation}
The base case and the survival sums then give
\[
 B_0=2C_V,\qquad B_0'=\Bigl(\frac{2\kstop}{\kstop-1}\Bigr)^2B_0,\qquad
 B_1=B_0'+\frac{(1+\vartheta^{-1})C_G}{1-\bar q_\star}.
\]

\paragraph{Reconstruction.}
With $a_d$ from \eqref{eq:ad}, the survival sums, the late-event bound
and the layered bound hold with
\begin{align}
 K_1&=\frac{1}{\kstop^2(1-\bar q_\star)}
 \Bigl(B_0'+\frac{(1+\vartheta^{-1})C_G\bar D}{e\,c_G}\Bigr),
 \label{eq:K1}\\
 K_2&=2B_{\rm max}\Bigl[\beta\,\bar D^{1/\beta}(2e\beta c_G)^{-1/\beta}
 +(J(C_b+1))^{1/\beta}\,\bar D^{1+1/\beta}
 \Bigl(\frac{1+1/\beta}{e\,a_d}\Bigr)^{1+1/\beta}\Bigr],\label{eq:K2}\\
 K_3&=\frac{(2\kstop^2+10)J(C_b+1)\bar D\,B_1}
           {e\,c_G\,\kstop^2}
 +2B_{\rm max}\Bigl[1+\bigl(J(C_b+1)\bar D\bigr)^{1/\beta}
          (e\beta a_d)^{-1/\beta}\Bigr].\label{eq:K3}
\end{align}

\paragraph{Main theorem and uniformity.}
An admissible choice in Theorem~\ref{thm:main} is
\begin{equation}\label{eq:final-constants}
 \begin{gathered}
 C_E=(2\kstop^2+10)J(C_b+1),\qquad
 \widetilde C_E=(2\kstop^2+10)J(C_b+1)K_1+K_2+K_3,\\
 C_\tau=J(C_b+1),\qquad
 \widetilde C_\tau=J(C_b+1)K_1+\frac{K_3}{2\kstop g_{\mathrm d}}.
 \end{gathered}
\end{equation}
The construction of $N_c$ uses the actual $q_\star$, but its upper
bound uses $\bar q_\star=1-g_{\mathrm c}/2$. Hence the constants in
\eqref{eq:final-constants} depend on $\nu$ and $\gamma$ only through
$g_{\mathrm c}$, and $\widetilde C_\tau$ in addition through the factor
$g_{\mathrm d}^{-1}$ in its second term. They are nonincreasing in
$g_{\mathrm c}$, which enters through
$\vartheta^{-1}\le12/g_{\mathrm c}$, $J\le1+\frac12\log(4/g_{\mathrm c})$,
$C_b\le(36/g_{\mathrm c})^{2/(2r+1)}$, $\bar D\le2c_G/g_{\mathrm c}+2$ and
$(1-\bar q_\star)^{-1}=2/g_{\mathrm c}$; in the bound for $C_b$ we used
$\kstop-1\ge g_{\mathrm c}/2$ and $c_r<1$, which holds since $r<1/2$.

The gap $g_{\mathrm c}$ is bounded below in terms of $\kstop$ alone. Indeed,
$\nu\le1/4$ and $\kstop>1/(1-\nu)$ give
$\nu/(1-\nu)\le\min\{1/3,\kstop-1\}$, hence
\begin{equation}\label{eq:gc-lower}
 g_{\mathrm c}\ge g_\ast(\kstop)
 :=\max\Bigl\{\frac23,\,2-\kstop\Bigr\}-\frac1{\kstop^2}>0;
\end{equation}
for $\kstop\le4/3$ positivity follows from
$2-\kstop-\kstop^{-2}=(\kstop-1)(1+\kstop-\kstop^2)/\kstop^2$, and for
$\kstop>4/3$ from $\kstop^{-2}<9/16$. Evaluating
\eqref{eq:final-constants} at $g_\ast(\kstop)$ in place of $g_{\mathrm c}$
therefore gives admissible constants $C_E$, $\widetilde C_E$, $C_\tau$,
$C_\tau'=J(C_b+1)K_1$ and $C_\tau''=K_3/(2\kstop)$ that depend only on
$r,\beta,C_\lambda,\kstop$, as stated in Theorem~\ref{thm:main}. The
bound \eqref{eq:gc-lower} is attained as $\nu\to\min\{1/4,1-1/\kstop\}$.

At $g_\ast(\kstop)$, the constants $C_E$ and $C_\tau$ depend only on $r$
and $\kstop$. As $\kstop\downarrow1$, the lower bound behaves like
$\kstop-1$, and the estimates above show that $C_E$ and $C_\tau$ grow like
$(\kstop-1)^{-2/(2r+1)}\log(1/(\kstop-1))$, through $C_b$ and $J$. This
is the power produced by the classical argument for Landweber iteration,
which bounds the stopping index by $(c_r/(\kstop-1))^{2/(2r+1)}\tprior/\gamma$
up to rounding and the propagated noise by the square root of the
corresponding time. The sampling constants $\widetilde C_E$, $C_\tau'$
and $C_\tau''$ grow at most like $(\kstop-1)^{-4-2/(2r+1)}$ up to a
factor $\log^2(1/(\kstop-1))$. We have not tried to optimize these powers.

More precisely, fix $r,\beta,C_\lambda$ satisfying the stated
conditions. For every finite $n,d$, every deterministic design,
initialization and observation noise satisfying the assumptions,
and every $\gamma,\rho,\delta>0$ and $\kstop$ with
$\nu\le1/4$, $\kstop>1/(1-\nu)$, $\delta\le\rho/e$ and $\gamma\le\tprior$,
the bounds hold with the constants above. No unquantified
small-noise condition is used.

\bibliographystyle{plain}
\bibliography{references}

\begin{thebibliography}{10}

\bibitem{ali2020implicit}
Alnur Ali, Edgar Dobriban, and Ryan~J. Tibshirani.
\newblock The implicit regularization of stochastic gradient flow for least squares.
\newblock In {\em Proceedings of the 37th International Conference on Machine Learning (ICML)}, volume 119 of {\em Proceedings of Machine Learning Research}, pages 233--244, 2020.

\bibitem{ali2019continuous}
Alnur Ali, J.~Zico Kolter, and Ryan~J. Tibshirani.
\newblock A continuous-time view of early stopping for least squares regression.
\newblock In {\em Proceedings of the 22nd International Conference on Artificial Intelligence and Statistics (AISTATS)}, volume~89 of {\em PMLR}, pages 1370--1378, 2019.

\bibitem{bach2013nonstrongly}
Francis Bach and Eric Moulines.
\newblock Non-strongly-convex smooth stochastic approximation with convergence rate {$O(1/n)$}.
\newblock In {\em Advances in Neural Information Processing Systems}, volume~26, pages 773--781, 2013.

\bibitem{bajpai2026hankeraus}
Harshit Bajpai, Gaurav Mittal, and Ankik~Kumar Giri.
\newblock {Hanke--Raus} heuristic rule for iteratively regularized stochastic gradient descent.
\newblock {\em Applied Numerical Mathematics}, 227:276--298, 2026.

\bibitem{bakushinskii1984remarks}
A.~B. Bakushinski{\u\i}.
\newblock Remarks on the choice of regularization parameter from quasioptimality and relation tests.
\newblock {\em Zh. Vychisl. Mat. i Mat. Fiz.}, 24(8):1258--1259, 1984.

\bibitem{bartlett2020benign}
Peter~L. Bartlett, Philip~M. Long, G{\'a}bor Lugosi, and Alexander Tsigler.
\newblock Benign overfitting in linear regression.
\newblock {\em Proceedings of the National Academy of Sciences}, 117(48):30063--30070, 2020.

\bibitem{becker2011regularization}
S.~M.~A. Becker.
\newblock Regularization of statistical inverse problems and the {B}akushinski{\u\i} veto.
\newblock {\em Inverse Problems}, 27(11):115010, 2011.

\bibitem{blanchard2018optimal}
Gilles Blanchard, Marc Hoffmann, and Markus Rei{\ss}.
\newblock Optimal adaptation for early stopping in statistical inverse problems.
\newblock {\em SIAM/ASA Journal on Uncertainty Quantification}, 6(3):1043--1075, 2018.

\bibitem{blanchard2012discrepancy}
Gilles Blanchard and Peter Math{\'e}.
\newblock Discrepancy principle for statistical inverse problems with application to conjugate gradient iteration.
\newblock {\em Inverse Problems}, 28(11):115011, 2012.

\bibitem{bottou2018}
L{\'e}on Bottou, Frank~E. Curtis, and Jorge Nocedal.
\newblock Optimization methods for large-scale machine learning.
\newblock {\em SIAM Review}, 60(2):223--311, 2018.

\bibitem{bubba2026linesearch}
Tatiana~A. Bubba, Elena Morotti, Federica Porta, Valeria Ruggiero, and Ilaria Trombini.
\newblock A line-search-based stochastic gradient method for {3D} computed tomography.
\newblock {\em arXiv preprint arXiv:2605.12085}, 2026.

\bibitem{buhlmann2003boosting}
Peter B{\"u}hlmann and Bin Yu.
\newblock Boosting with the {$L_2$} loss: regression and classification.
\newblock {\em Journal of the American Statistical Association}, 98(462):324--339, 2003.

\bibitem{candes2007sparsity}
Emmanuel Cand{\`e}s and Justin Romberg.
\newblock Sparsity and incoherence in compressive sampling.
\newblock {\em Inverse Problems}, 23(3):969--985, 2007.

\bibitem{candes2009exact}
Emmanuel~J. Cand{\`e}s and Benjamin Recht.
\newblock Exact matrix completion via convex optimization.
\newblock {\em Foundations of Computational Mathematics}, 9(6):717--772, 2009.

\bibitem{caponnetto2007optimal}
Andrea Caponnetto and Ernesto De~Vito.
\newblock Optimal rates for the regularized least-squares algorithm.
\newblock {\em Foundations of Computational Mathematics}, 7(3):331--368, 2007.

\bibitem{celisse2021analyzing}
Alain Celisse and Martin Wahl.
\newblock Analyzing the discrepancy principle for kernelized spectral filter learning algorithms.
\newblock {\em Journal of Machine Learning Research}, 22(76):1--59, 2021.
\newblock arXiv:2004.08436.

\bibitem{chambolle2018stochastic}
Antonin Chambolle, Matthias~J. Ehrhardt, Peter Richt{\'a}rik, and Carola-Bibiane Sch{\"o}nlieb.
\newblock Stochastic primal-dual hybrid gradient algorithm with arbitrary sampling and imaging applications.
\newblock {\em SIAM Journal on Optimization}, 28(4):2783--2808, 2018.

\bibitem{dieuleveut2016nonparametric}
Aymeric Dieuleveut and Francis Bach.
\newblock Nonparametric stochastic approximation with large step-sizes.
\newblock {\em The Annals of Statistics}, 44(4):1363--1399, 2016.

\bibitem{dieuleveut2017harder}
Aymeric Dieuleveut, Nicolas Flammarion, and Francis Bach.
\newblock Harder, better, faster, stronger convergence rates for least-squares regression.
\newblock {\em Journal of Machine Learning Research}, 18(101):1--51, 2017.

\bibitem{ehrhardt2025guide}
Matthias~J. Ehrhardt, {\v Z}eljko Kereta, Jingwei Liang, and Junqi Tang.
\newblock A guide to stochastic optimisation for large-scale inverse problems.
\newblock {\em Inverse Problems}, 41(5):053001, 2025.

\bibitem{ehrhardt2019faster}
Matthias~J. Ehrhardt, Pawel Markiewicz, and Carola-Bibiane Sch{\"o}nlieb.
\newblock Faster {PET} reconstruction with non-smooth priors by randomization and preconditioning.
\newblock {\em Physics in Medicine \& Biology}, 64(22):225019, 2019.

\bibitem{engl1996}
Heinz~W. Engl, Martin Hanke, and Andreas Neubauer.
\newblock {\em Regularization of Inverse Problems}, volume 375 of {\em Mathematics and Its Applications}.
\newblock Kluwer Academic Publishers, Dordrecht, 1996.

\bibitem{gower2015randomized}
Robert~M. Gower and Peter Richt{\'a}rik.
\newblock Randomized iterative methods for linear systems.
\newblock {\em SIAM Journal on Matrix Analysis and Applications}, 36(4):1660--1690, 2015.

\bibitem{gu2026posteriori}
Ruixue Gu and Qinian Jin.
\newblock On a posteriori stopping rules of adaptive stochastic heavy ball method for ill-posed problems.
\newblock {\em arXiv preprint arXiv:2605.13144}, 2026.

\bibitem{hansen1998rank}
Per~Christian Hansen.
\newblock {\em Rank-Deficient and Discrete Ill-Posed Problems: Numerical Aspects of Linear Inversion}.
\newblock SIAM, Philadelphia, 1998.

\bibitem{harrach2020beyond}
Bastian Harrach, Tim Jahn, and Roland Potthast.
\newblock Beyond the {B}akushinskii veto: regularising linear inverse problems without knowing the noise distribution.
\newblock {\em Numerische Mathematik}, 145(3):581--603, 2020.

\bibitem{harrach2023regularizing}
Bastian Harrach, Tim Jahn, and Roland Potthast.
\newblock Regularizing linear inverse problems under unknown non-{G}aussian white noise allowing repeated measurements.
\newblock {\em IMA Journal of Numerical Analysis}, 43(1):443--500, 2023.

\bibitem{huang2025early}
Jing Huang, Qinian Jin, Xiliang Lu, and Liuying Zhang.
\newblock On early stopping of stochastic mirror descent method for ill-posed inverse problems.
\newblock {\em Numerische Mathematik}, 157(2):539--571, 2025.

\bibitem{jahn2023noise}
Tim Jahn.
\newblock Noise level free regularization of general linear inverse problems under unconstrained white noise.
\newblock {\em SIAM/ASA Journal on Uncertainty Quantification}, 11(2):591--615, 2023.

\bibitem{jahn2020discrepancy}
Tim Jahn and Bangti Jin.
\newblock On the discrepancy principle for stochastic gradient descent.
\newblock {\em Inverse Problems}, 36(9):095009, 2020.

\bibitem{jahn2024early}
Tim Jahn and Bangti Jin.
\newblock Early stopping of untrained convolutional neural networks.
\newblock {\em SIAM Journal on Imaging Sciences}, 17(4):2331--2361, 2024.

\bibitem{jahn2026sde}
Tim Jahn, Loucas Pillaud-Vivien, and Adrien Schertzer.
\newblock Minimax-optimal early stopping for continuous-time {SGD} via the discrepancy principle.
\newblock {\em arXiv preprint arXiv:2609.11273v1}, 2026.

\bibitem{jain2018parallelizing}
Prateek Jain, Sham~M. Kakade, Rahul Kidambi, Praneeth Netrapalli, and Aaron Sidford.
\newblock Parallelizing stochastic gradient descent for least squares regression: mini-batching, averaging, and model misspecification.
\newblock {\em Journal of Machine Learning Research}, 18(223):1--42, 2018.

\bibitem{jiang2005maxima}
Tiefeng Jiang.
\newblock Maxima of entries of {H}aar distributed matrices.
\newblock {\em Probability Theory and Related Fields}, 131(1):121--144, 2005.

\bibitem{jin2019regularizing}
Bangti Jin and Xiliang Lu.
\newblock On the regularizing property of stochastic gradient descent.
\newblock {\em Inverse Problems}, 35(1):015004, 2019.

\bibitem{jin2025survey}
Bangti Jin, Yimin Xia, and Zehui Zhou.
\newblock On the regularizing property of stochastic iterative methods for solving inverse problems.
\newblock In {\em Handbook of Numerical Analysis}, volume~26, pages 211--272. Elsevier, 2025.

\bibitem{jinzhou2026dp}
Bangti Jin and Zehui Zhou.
\newblock Early stopping of stochastic variance reduced gradient for linear inverse problems by the discrepancy principle.
\newblock {\em arXiv preprint arXiv:2609.12396}, 2026.

\bibitem{jinzhou2026svrg}
Bangti Jin and Zehui Zhou.
\newblock On the convergence of stochastic variance reduced gradient for linear inverse problems.
\newblock {\em Inverse Problems}, 42(4):045006, 2026.

\bibitem{jin2020convergence}
Bangti Jin, Zehui Zhou, and Jun Zou.
\newblock On the convergence of stochastic gradient descent for nonlinear ill-posed problems.
\newblock {\em SIAM Journal on Optimization}, 30(2):1421--1450, 2020.

\bibitem{jin2021saturation}
Bangti Jin, Zehui Zhou, and Jun Zou.
\newblock On the saturation phenomenon of stochastic gradient descent for linear inverse problems.
\newblock {\em SIAM/ASA Journal on Uncertainty Quantification}, 9(4):1553--1588, 2021.

\bibitem{jin2022svrg}
Bangti Jin, Zehui Zhou, and Jun Zou.
\newblock An analysis of stochastic variance reduced gradient for linear inverse problems.
\newblock {\em Inverse Problems}, 38(2):025009, 2022.

\bibitem{jinqchen2025svrg}
Qinian Jin and Liuhong Chen.
\newblock Stochastic variance reduced gradient method for linear ill-posed inverse problems.
\newblock {\em Inverse Problems}, 41(5):055014, 2025.

\bibitem{jinlutian2026}
Qinian Jin, Xiliang Lu, and Lin Tian.
\newblock On convergence rates of stochastic gradient descent for linear inverse problems.
\newblock {\em Numerische Mathematik}, 2026.
\newblock Published online 8 September 2026.

\bibitem{li2017stochastic}
Qianxiao Li, Cheng Tai, and Weinan E.
\newblock Stochastic modified equations and adaptive stochastic gradient algorithms.
\newblock In {\em Proceedings of the 34th International Conference on Machine Learning (ICML)}, volume~70 of {\em PMLR}, pages 2101--2110, 2017.

\bibitem{li2019stochastic}
Qianxiao Li, Cheng Tai, and Weinan E.
\newblock Stochastic modified equations and dynamics of stochastic gradient algorithms {I}: mathematical foundations.
\newblock {\em Journal of Machine Learning Research}, 20(40):1--47, 2019.

\bibitem{lin2017optimal}
Junhong Lin and Lorenzo Rosasco.
\newblock Optimal rates for multi-pass stochastic gradient methods.
\newblock {\em Journal of Machine Learning Research}, 18(97):1--47, 2017.

\bibitem{lu2022stochastic}
Shuai Lu and Peter Math{\'e}.
\newblock Stochastic gradient descent for linear inverse problems in {H}ilbert spaces.
\newblock {\em Mathematics of Computation}, 91(336):1763--1788, 2022.

\bibitem{ma2015convergence}
Anna Ma, Deanna Needell, and Aaditya Ramdas.
\newblock Convergence properties of the randomized extended {G}auss--{S}eidel and {K}aczmarz methods.
\newblock {\em SIAM Journal on Matrix Analysis and Applications}, 36(4):1590--1604, 2015.

\bibitem{martin2025pnp}
S{\'e}gol{\`e}ne Martin, Anne Gagneux, Paul Hagemann, and Gabriele Steidl.
\newblock {PnP-Flow}: Plug-and-play image restoration with flow matching.
\newblock In {\em International Conference on Learning Representations (ICLR)}, 2025.

\bibitem{melkman1979optimal}
Avraham~A. Melkman and Charles~A. Micchelli.
\newblock Optimal estimation of linear operators in {H}ilbert spaces from inaccurate data.
\newblock {\em SIAM Journal on Numerical Analysis}, 16(1):87--105, 1979.

\bibitem{micchelli1977survey}
Charles~A. Micchelli and Theodore~J. Rivlin.
\newblock A survey of optimal recovery.
\newblock In Charles~A. Micchelli and Theodore~J. Rivlin, editors, {\em Optimal Estimation in Approximation Theory}, pages 1--54. Plenum Press, New York, 1977.

\bibitem{morozov1966solution}
V.~A. Morozov.
\newblock On the solution of functional equations by the method of regularization.
\newblock {\em Soviet Mathematics Doklady}, 7:414--417, 1966.
\newblock Russian original: Dokl. Akad. Nauk SSSR 167(3):510--512.

\bibitem{needell2010randomized}
Deanna Needell.
\newblock Randomized {K}aczmarz solver for noisy linear systems.
\newblock {\em BIT Numerical Mathematics}, 50(2):395--403, 2010.

\bibitem{needell2016stochastic}
Deanna Needell, Nathan Srebro, and Rachel Ward.
\newblock Stochastic gradient descent, weighted sampling, and the randomized {K}aczmarz algorithm.
\newblock {\em Mathematical Programming}, 155(1--2):549--573, 2016.

\bibitem{paquette2021sgd}
Courtney Paquette, Kiwon Lee, Fabian Pedregosa, and Elliot Paquette.
\newblock {SGD} in the large: Average-case analysis, asymptotics, and stepsize criticality.
\newblock In {\em Proceedings of the 34th Conference on Learning Theory (COLT)}, volume 134 of {\em Proceedings of Machine Learning Research}, pages 3548--3626, 2021.

\bibitem{paquette2025homogenization}
Courtney Paquette, Elliot Paquette, Ben Adlam, and Jeffrey Pennington.
\newblock Homogenization of {SGD} in high-dimensions: exact dynamics and generalization properties.
\newblock {\em Mathematical Programming}, 214:1--90, 2025.

\bibitem{pillaudvivien2018statistical}
Loucas Pillaud-Vivien, Alessandro Rudi, and Francis Bach.
\newblock Statistical optimality of stochastic gradient descent on hard learning problems through multiple passes.
\newblock In {\em Advances in Neural Information Processing Systems 31 (NeurIPS)}, 2018.
\newblock arXiv:1805.10074.

\bibitem{raskutti2014early}
Garvesh Raskutti, Martin~J. Wainwright, and Bin Yu.
\newblock Early stopping and non-parametric regression: an optimal data-dependent stopping rule.
\newblock {\em Journal of Machine Learning Research}, 15(11):335--366, 2014.

\bibitem{robbins1951stochastic}
Herbert Robbins and Sutton Monro.
\newblock A stochastic approximation method.
\newblock {\em Annals of Mathematical Statistics}, 22(3):400--407, 1951.

\bibitem{schertzer2024stochastic}
Adrien Schertzer and Loucas Pillaud-Vivien.
\newblock Stochastic differential equations models for least-squares stochastic gradient descent.
\newblock {\em Journal of Machine Learning Research}, 27(75):1--42, 2026.

\bibitem{sirignano2017stochastic}
Justin Sirignano and Konstantinos Spiliopoulos.
\newblock Stochastic gradient descent in continuous time.
\newblock {\em SIAM Journal on Financial Mathematics}, 8(1):933--961, 2017.

\bibitem{strohmer2009randomized}
Thomas Strohmer and Roman Vershynin.
\newblock A randomized {K}aczmarz algorithm with exponential convergence.
\newblock {\em Journal of Fourier Analysis and Applications}, 15(2):262--278, 2009.

\bibitem{tang2020practicality}
Junqi Tang, Karen Egiazarian, Mohammad Golbabaee, and Mike Davies.
\newblock The practicality of stochastic optimization in imaging inverse problems.
\newblock {\em IEEE Transactions on Computational Imaging}, 6:1471--1485, 2020.

\bibitem{varre2021last}
Aditya~Vardhan Varre, Loucas Pillaud-Vivien, and Nicolas Flammarion.
\newblock Last iterate convergence of {SGD} for least-squares in the interpolation regime.
\newblock In {\em Advances in Neural Information Processing Systems 34 (NeurIPS)}, pages 21581--21591, 2021.

\bibitem{vogel2002computational}
Curtis~R. Vogel.
\newblock {\em Computational Methods for Inverse Problems}.
\newblock SIAM, Philadelphia, 2002.

\bibitem{yao2007early}
Yuan Yao, Lorenzo Rosasco, and Andrea Caponnetto.
\newblock On early stopping in gradient descent learning.
\newblock {\em Constructive Approximation}, 26(2):289--315, 2007.

\bibitem{zhang2005boosting}
Tong Zhang and Bin Yu.
\newblock Boosting with early stopping: convergence and consistency.
\newblock {\em The Annals of Statistics}, 33(4):1538--1579, 2005.

\bibitem{zhu2026restarted}
Kai Zhu, Qinian Jin, Wei Wang, and Min Zhong.
\newblock Restarted adaptive stochastic {N}esterov method with early stopping for ill-posed inverse problems.
\newblock {\em Numerische Mathematik}, 2026.
\newblock Online first, doi:10.1007/s00211-026-01575-x.

\bibitem{zouzias2013randomized}
Anastasios Zouzias and Nikolaos~M. Freris.
\newblock Randomized extended {K}aczmarz for solving least squares.
\newblock {\em SIAM Journal on Matrix Analysis and Applications}, 34(2):773--793, 2013.

\end{thebibliography}

\end{document}